\documentclass{siamart0516modified}

\usepackage{lipsum}
\usepackage{amsfonts}
\usepackage{graphicx}
\usepackage{epstopdf}
\usepackage{algorithmic}
\ifpdf
  \DeclareGraphicsExtensions{.eps,.pdf,.png,.jpg}
\else
  \DeclareGraphicsExtensions{.eps}
\fi

\usepackage[normalem]{ulem}
\usepackage{enumitem}
\usepackage{placeins}

\newcommand{\TheAuthors}{N. Alger, U. Villa, T. Bui-Thanh, and O. Ghattas}

\newcommand{\ShortTitle}{A data scalable KKT preconditioner}
\newcommand{\LongTitle}{A data scalable augmented Lagrangian KKT
  preconditioner for large scale inverse problems}

\newcommand{\nor}[1]{\left\| #1 \right\|}

\newcommand{\cond}[1]{\operatorname{cond}\left( #1 \right)}

\newsiamthm{thm}{Theorem}
\newsiamthm{cor}{Corollary}
\newsiamthm{prop}{Proposition}
\newsiamthm{lem}{Lemma}
\newsiamthm{example}{Example}
\newsiamthm{defn}{Definition}
\newsiamthm{assume}{Assumption}
\newsiamthm{condition}{Condition}

\newcommand{\concrete}[1]{\boldsymbol{\mathbf{#1}}}
\newcommand{\adjoint}{\eta}
\newcommand{\parametermode}{\phi}
\newcommand{\parametereigenvalue}{d}

\newcommand{\basisoperator}{\Theta}
\newcommand{\noisevector}{\zeta}

\newcommand{\reglower}{\mu}
\newcommand{\regupper}{\nu}
\newcommand{\errover}{e_q}
\newcommand{\errunder}{e_\noisevector}

\newcommand{\ind}{k}

\newcommand{\regassumption}{Assumption \ref{condition:appropriate_regularization}}
\newcommand{\nobs}{n_\text{obs}}

\newcommand{\qpre}{p}

\newcommand{\upre}{v}

\newcommand{\etapre}{\xi}

\newcommand{\xpre}{z}

\graphicspath{{../figures/}}

\headers{\ShortTitle}{\TheAuthors}

\title{{\LongTitle}\thanks{This work was funded by DOE grants
    DE-SC0010518 and DE-SC0009286, AFOSR grant FA9550-12-1-0484, and NSF grant CBET-1508713}}

\author{
  Nick Alger\thanks{Institute for Computational Engineering and Sciences, The University of Texas at Austin, Austin, TX
    (\email{nalger@ices.utexas.edu}, \email{uvilla@ices.utexas.edu}, \email{omar@ices.utexas.edu}).}
  \and
  Umberto Villa\footnotemark[2]
  \and
  Tan Bui-Thanh\thanks{Department of Aerospace Engineering and Engineering Mechanics, and Institute for Computational Engineering and Sciences, The University of Texas at Austin, Austin, TX
    (\email{tanbui@ices.utexas.edu}).}
  \and
  Omar Ghattas\footnotemark[2]
}

\usepackage{amsopn}

\ifpdf

\hypersetup{
  pdftitle={\LongTitle},
  pdfauthor={\TheAuthors}
}

\fi

\begin{document}

\maketitle

\begin{abstract}
Current state of the art preconditioners for the reduced Hessian and the Karush-Kuhn-Tucker (KKT) operator for large
scale inverse problems are typically based on approximating the
reduced Hessian with the regularization operator. However, the quality
of this approximation degrades with increasingly informative 
observations or data. Thus the best case scenario from a scientific standpoint (fully
informative data) is the worse case scenario from a computational
perspective. In this paper we present an augmented Lagrangian-type
preconditioner based on a block diagonal approximation of the
augmented upper left block of the KKT operator. The preconditioner
requires solvers for two 
linear subproblems that arise in the augmented KKT operator, which we
expect to be much easier to precondition than the reduced Hessian.
Analysis of the spectrum of the preconditioned KKT operator indicates
that the preconditioner is effective when the regularization is chosen
appropriately. In particular, it is effective when the regularization
does not over-penalize highly informed parameter modes and does not
under-penalize uninformed modes. Finally, we present a numerical study
for a large data/low noise Poisson source inversion problem,
demonstrating the effectiveness of the preconditioner. In this
example, three MINRES iterations on the KKT system with our
preconditioner results in a reconstruction with better accuracy than
50 iterations of CG on the reduced Hessian system with regularization
preconditioning.

\end{abstract}



\section{Introduction}
Here we focus on data scalable Karush-Kuhn-Tucker (KKT) preconditioners
for large-scale linear\footnote{The preconditioner presented here is
  also applicable to nonlinear inverse problems, which give rise to
  linear systems of the form we address here at each iteration of a
  (Gauss) Newton or sequential quadratic programming method.}  inverse
problems in which one seeks to reconstruct a parameter field from
observations of an associated state variable. Specifically, suppose we
have observations $y$ of a state variable $u$ that have been corrupted
by some noise $\noisevector$,
\begin{equation}
\label{eqn:obs_with_noise}
y=Bu + \noisevector,
\end{equation}
where $B$ is a linear operator encoding the action of the observation
process (i.e., the observation operator). Further, let the state variable
$u$ depend on a parameter $q$ through a linear state equation,
\begin{equation}
\label{eqn:state_equation}
Tq + Au = f,
\end{equation}
where $A$ is the forward operator, $T$ maps the
  parameter to the residual of the state equation, and $f$ is the
  known right hand side source. We seek to reconstruct the parameter
$q$ based on the observations $y$.
Under an independent and identically distributed (i.i.d.)
Gaussian noise model,\footnote{The least squares formulation in \eqref{eqn:linear_optimization_problem} also applies to general (non-i.i.d.) Gaussian noise models after re-weighting the data misfit equation \eqref{eqn:obs_with_noise} by the inverse square root of the noise covariance.} this inverse problem naturally leads to a least squares
optimization problem of the following form,
\begin{equation}
\label{eqn:linear_optimization_problem}
\begin{aligned}
\min_{q,u} &\quad \frac{1}{2}\nor{Bu-y}^2 + \frac{\alpha}{2}\nor{Rq}^2 \\
\text{such that} &\quad Tq + Au = f,
\end{aligned}
\end{equation}
where $R$ is a suitably chosen regularization operator and $\alpha$ is
a regularization parameter. We focus on the computationally difficult
case in which the observations $y$
are highly informative about the parameter $q$, and the regularization
parameter $\alpha$ is correspondingly small. The KKT system expressing
the first order necessary condition for an optimal solution of 
\eqref{eqn:linear_optimization_problem} is
\begin{equation}
\label{eqn:KKT_matrix_full}
\underbrace{\begin{bmatrix}\alpha R^*R & & T^* \\
& B^*B & A^* \\
T & A
\end{bmatrix}}_{K} \begin{bmatrix}q \\ u \\ \adjoint \end{bmatrix} = \begin{bmatrix}0 \\ B^* y \\ f\end{bmatrix}.
\end{equation}
Here all operators are linear, possibly unbounded,
maps between suitable Hilbert spaces. The symbols $A^*,B^*, R^*$, and $T^*$ denote the adjoints (in the Hilbert space sense) of $A,B,R$, and $T$, respectively, and $\adjoint$ denotes the adjoint variable (or Lagrange
multiplier) for the state equation. More details regarding our setting
and notation are provided in Section \ref{sec:notation}. We denote
the KKT operator by $K$. For large-scale problems, direct
factorization of $K$ is not computationally feasible due to both
memory and CPU time considerations, and iterative methods must be
used. Krylov methods, MINRES \cite{PaigeSaunders75} in particular, are
the gold standard for iteratively solving this kind of large scale
symmetric indefinite system. The performance of the MINRES method
strongly depends on the clustering of the spectrum of the
preconditioned operator, the more clustered the better
\cite{WathenFischerSilvester95}. In this paper we propose clustering
the spectrum of the KKT operator by using the following block diagonal
preconditioner,
\begin{equation}
\label{eqn:preconditioner}
P:=\begin{bmatrix}
\alpha R^*R + \rho T^*T \\
& B^*B + \rho A^*A \\
& & \frac{1}{\rho} I
\end{bmatrix},
\end{equation}
where $I$ denotes the identity map associated with the appropriate
inner product (in the computations, a mass matrix). We further propose
choosing $\rho = \sqrt{\alpha}$ based on theoretical results and
numerical evidence. In our theory and numerical experiments we assume
that $A$ and $R$ are invertible maps. Although the application of
preconditioner \eqref{eqn:preconditioner} and the abstract theory we
present in Section \ref{sec:theory} do not depend on invertibility of
$T$, much of the intuition behind the assumptions of the theory is
lacking in the case where $T$ is non-invertible. Remedies for
this case are the subject of ongoing research. 
While existing data scalable KKT preconditioners usually require
regularization operators $R$ that are spectrally equivalent to the
identity,\footnote{A review of existing work is presented in Section
  \ref{sec:review_of_existing_work}. In particular, see Sections
  \ref{sec:adjoint_schur} and \ref{sec:block_scaling}.} our
preconditioner \eqref{eqn:preconditioner} performs well even if $R$ is
a discretization of an unbounded operator (e.g., Laplacian
regularization).

\subsection{Overview of results}

In Section \ref{sec:cond_bound} we prove that, using our
preconditioner \eqref{eqn:preconditioner}, the symmetrically
preconditioned KKT operator satisfies the  condition number
bound 
\begin{equation*}
\cond{P^{-1/2} K P^{-1/2}} \le \frac{3}{(1-\beta)\delta}
\,\, ,
\end{equation*}
where $\delta$ and $\beta$ are bounds on the eigenvalues of the
arithmetic and geometric means of certain damped projectors.\footnote{The condition number provides an upper bound on the required number of MINRES iterations. An even sharper bound could by obtained by characterizing all four extreme eigenvalues (endpoints of the positive and negative intervals in which the eigenvalues reside) of the preconditioned system \cite{ElmanSilvesterWathen14}.} 
Based on
the nature of the damped projectors, we expect these eigenvalue bounds
to be satisfied with good constants $\delta$ and $\beta$ for inverse
problems that are \emph{appropriately regularized}. By ``appropriately regularized,'' we mean that the regularization is chosen so that components of the
parameter that are highly informed by the data are not over-penalized,
and components of the parameter that are poorly informed by the data are
not under-penalized. In Section \ref{sec:source_filter} we derive
quantitative bounds on $\delta$ and $\beta$ for the special case of
source inversion problems with spectral filtering regularization. When
the regularization is chosen appropriately, these bounds are
independent of the mesh size and of the information content in the data. 

In Section \ref{sec:numerical_results} we numerically demonstrate the
effectiveness of the preconditioner on a Poisson source inversion
problem with highly informative data and Laplacian
regularization. Preconditioning the KKT system with our preconditioner
results in greater accuracy in three MINRES iterations than the
widely-used regularization preconditioning on the reduced Hessian
system achieves in 50 conjugate gradient iterations. Even though the
regularization is not a spectral filter,
our preconditioner still exhibits mesh independence and good scalability with
respect to a decrease in the regularization parameter
by 10 orders of magnitude. 
As suggested by our theory, we see that the performance
of the preconditioner in the small regularization regime actually
improves as more data is included in the inversion.

\subsection{Desirable properties of KKT preconditioners for inverse problems}

To evaluate the quality of a 
KKT preconditioner, it is useful to consider its performance with
respect to the following desired properties: 
\begin{enumerate}[label=(\alph{enumi})]

\item \textbf{Problem generality:} A KKT preconditioner exhibits
  problem generality if it applies to a wide variety of inverse
  problems. \label{enum:prob_gen}

\item \textbf{Efficient solvers for preconditioner subproblems:} If
  applying the inverse of the preconditioner to a vector involves
  solving subproblems, efficient solvers for those subproblems are
  required.  \label{enum:eff_subprob}

\item \textbf{Mesh scalability:} Finite dimensional inverse problems
  often arise from discretizations of infinite dimensional inverse
  problems. Preconditioners for such problems are {\em mesh scalable}
  if the effectiveness of the preconditioner (as measured in terms of
  either the condition number of the preconditioned KKT operator, the
  clustering of the spectrum of the preconditioned KKT operator, or
  the number of Krylov iterations required to converge to a fixed
  tolerance) does not degrade substantially as the meshes used to
  discretize the problem are refined.  \label{enum:mesh_scale}

\item \textbf{Regularization robustness:} KKT preconditioners are
  {\em regularization robust} if their effectiveness does not degrade
  substantially as the regularization parameter $\alpha$ is made
  smaller. \label{enum:reg_dependence} 

\item \textbf{Data scalability:} KKT preconditioners are {\em data
  scalable} if their effectiveness does not degrade substantially as
  more data---or rather, more informative data---are included in
  the inverse problem. \label{enum:data_dependence}

\end{enumerate}
Currently there is no known preconditioner 
that exhibits uniformly good performance with respect to all of these
properties. 
In this paper, we provide a combination of theoretical results and numerical
evidence demonstrating that our preconditioner provides substantial
improvements over existing preconditioners, especially with respect to
problem generality and data scalability. 

Within the scope of inverse problems, we view the goal of
robustness to arbitrarily-chosen values of the regularization
parameter, \ref{enum:reg_dependence}, to be 
unwarranted and unnecessarily restrictive. In particular, for properly
regularized inverse problems the regularization operator and
regularization parameter are not arbitrary. Rather, they are chosen in
response to the data available in the problem: that is, to constrain
parameter modes that are not informed by the data, while minimally
modifying components of the parameter that are informed by the
data. Thus it is important that the preconditioner performs well as
the informative content of the data increases while the regularization
parameter decreases correspondingly. However, it is not important for
the preconditioner to perform well in the under-regularized regime in
which the regularization parameter is small but the data are
uninformative. In this under-regularized regime, a good preconditioner
would simply accelerate convergence to noise, i.e., more rapid
solution of the wrong problem. Instead, we advocate designing
preconditioners that perform well with increasingly informative data,
\ref{enum:data_dependence}, for which the regularization parameter is
considered a dependent parameter chosen so that the inverse problem is
neither substantially over- nor under-regularized. This extra
flexibility permits design of the preconditioner to better address the
entire set of desired properties
\ref{enum:prob_gen}--\ref{enum:data_dependence}. 

Among data \ref{enum:data_dependence} and mesh scalable
\ref{enum:mesh_scale} preconditioners, ours is the most general
\ref{enum:prob_gen}. The subproblems that must be solved
\ref{enum:eff_subprob} while applying the preconditioner are of
similar difficulty to those encountered by existing (less general)
data-scalable preconditioners. What remains for our preconditioner to
fully satisfy all of the remaining desirable properties,
\ref{enum:prob_gen}, \ref{enum:eff_subprob}, \ref{enum:mesh_scale},
and \ref{enum:data_dependence}, is to generalize it to non-invertible
$T$. As mentioned above, this is ongoing research;
nevertheless, there are many inverse problems characterized by
invertible $T$ operators.  In addition to source inversion problems
(addressed in Sections \ref{sec:source_filter} and
\ref{sec:numerical_results}), coefficient inverse problems in which
the state and parameter share the same discretization often give rise
to invertible $T$.

\subsection{Review of existing work}
\label{sec:review_of_existing_work}

A wide variety of preconditioners for KKT operators similar to
\eqref{eqn:KKT_matrix_full} have been
developed in a number of different contexts including parameter
estimation, optimal control, PDE constrained optimization, optimal
design, and saddle point systems arising in mixed discretizations of
forward problems \cite{BenziGolubLiesen05, Choi12, MardalWinther10}. In the following subsections we discuss existing preconditioners based on the reduced Hessian (Section \ref{sec:reduced_hessian_literature}), the adjoint Schur complement (Section \ref{sec:adjoint_schur}), block scaling (Section \ref{sec:block_scaling}), and multigrid (Section \ref{sec:multigrid}).

We will see that existing preconditioners either scale poorly with increasing data and
decreasing regularization, or they only apply
to specific problems, or they make restrictive assumptions about the $B$, $R$, and $T$ operators. In particular, in the literature it is common to assume that the parameter and/or observation spaces
are $L^2$ spaces, and one or more of the operators $B$, $R$, and $T$ are spectrally equivalent to either identity maps ($I$), or restriction maps ($\Gamma$) that restrict functions to a subdomain. 
These assumptions on $B$, $R$ and $T$ may be inappropriate for the inverse problem at hand. For example, they prevent one from using observations of derived quantities such as flux, using smoothing Laplacian-like regularization,
and inverting for material coefficients. We will regularly note such assumptions by following references with a parenthetical expression. E.g., ``\cite{StollWathen10} ($L^2$, $R\approx I,~B\approx \Gamma,~T\approx-I$)'' means that the preconditioner in reference \cite{StollWathen10} assumes that the parameter and observation spaces are $L^2$ spaces, $R$ is spectrally equivalent to an identity map ($L^2$ regularization), $B$ is spectrally equivalent to a restriction map (direct observations of the state on a subdomain), and $T$ is spectrally equivalent to a negative identity map (the parameter enters the state equation on the right hand side as a source term).

\subsubsection{The reduced Hessian}
\label{sec:reduced_hessian_literature}
The reduced Hessian is the Hessian of the unconstrained reformulation
of optimization problem \eqref{eqn:linear_optimization_problem}, in
which the constraint is eliminated by viewing the state $u$ as an
implicit function of the parameter $q$ via solution of the state
equation. We discuss this reduced space problem in more detail in
Section \ref{sec:reduced_problem}. For linear inverse problems (as
considered in this paper), the reduced Hessian is equivalent to the
Schur complement of the KKT operator with respect to the parameter. In
other words, it is the operator remaining when the state and adjoint
variables (and corresponding equations) are solved for and eliminated
from the KKT system. Likewise, the KKT operator can be derived by
starting with the reduced Hessian, defining auxiliary variables, and
performing simple algebraic manipulations. Thus performing solves with
the reduced Hessian and performing KKT solves are
equivalent: if one can efficiently solve the former then one
can efficiently solve the later and vice versa. For this reason, a
popular class of methods for solving \eqref{eqn:KKT_matrix_full}
relies on approximations or preconditioners for the reduced Hessian
\cite{BirosGhattas05a, BirosGhattas05, HaberAscher01}.

The most popular class of general purpose preconditioners for the
reduced Hessian is based on approximating this operator with just the
regularization operator, and either neglecting the data misfit term or
dealing with it through some form of low rank approximation.  The
regularization is typically an elliptic operator and can be inverted
using multigrid or other standard techniques. Furthermore, for
ill-posed inverse problems the data misfit portion of the reduced
Hessian at the optimal solution is usually a compact operator in the
infinite-dimensional limit \cite{Bui-ThanhGhattas12a,
  Bui-ThanhGhattas12, Bui-ThanhGhattas13a, Vogel02}. Thus Krylov
methods preconditioned by the regularization operator usually yield
mesh independent, superlinear convergence rates\footnote{Here, by \emph{superlinear}, we mean that the norm of the error decays superlinearly with respect to the number of Krylov iterations.} \cite{AxelssonKaratson07, Fortuna79a, HerzogSachs15}. However, the importance of the
regularization term in the reduced Hessian decreases as the
regularization parameter is made smaller, and the importance of the
data misfit term increases as the informativeness of the data
increases. Indeed, the numerical rank of the data misfit portion of
the reduced Hessian is roughly the number of parameter modes that are
``informed'' by the data. In addition, the eigenvalues of the regularization
preconditioned Hessian are typically well-separated, which means that this approach will still require
large numbers of Krylov iterations on problems with highly informative data.
Thus, the best case scenario from a scientific
standpoint (highly informative data) is the worse case scenario from a
computational standpoint (large numbers of Krylov iterations
required).

Other problem-specific reduced Hessian solvers and preconditioners
have been developed using a diverse set of techniques including
analysis of the symbol of the reduced Hessian \cite{ArianTaasan99},
matrix probing \cite{DemanetLetourneauBoumalEtAl12}, approximate
sparsity in curvelet frames \cite{HerrmannMoghaddamPeyman08}, and
analytic expressions derived for model problems \cite{AdavaniBiros10, ArianIollo09, Flath13, FlathWilcoxAkcelikEtAl11, GholamiMangBiros16}. 

\subsubsection{Schur complement for the adjoint variable}
\label{sec:adjoint_schur}

In contrast to the approaches based on the reduced Hessian described
above (where the state and adjoint are eliminated), another class of
preconditioners of increasing interest in recent years is based on
block factorizations that eliminate the parameter and state, resulting
in a Schur complement operator for the adjoint variable.  This
approach requires one to design preconditioners for the objective
block (the 2$\times$2 block corresponding to $q$ and $u$ in
\eqref{eqn:KKT_matrix_full}) and for the Schur complement associated
with the adjoint variable. In the case of limited observations, the
objective block is singular and requires special handling; a common
approach is to add a small positive diagonal shift to the block. 

Mesh independent block diagonal preconditioners based on approximating
the objective block with mass matrices and the adjoint Schur
complement with $AA^*$ have been proposed for $L^2$ regularized
optimal control problems with the Poisson equation as a constraint
and a control objective targeting the state variable directly \cite{ReesDollarWathen10, ReesStollWathen10} ($L^2$, $R\approx I, B\approx I, T\approx -I$), and extended
to problems with parabolic PDE constraints and limited observations \cite{StollWathen10} ($L^2$, $R\approx I, B\approx \Gamma, T\approx-I$). 
More nuanced approximations of the Schur
complement have been shown to yield robustness with respect to the
regularization parameter for problems in the elliptic case in \cite{PearsonWathen12} ($L^2$, $R\approx I, B \approx I, T\approx -I$) and the parabolic case in \cite{PearsonStollWathen12} ($L^2$, $R\approx \Gamma, B\approx \Gamma, T\approx -\Gamma$). Regularization robust adjoint Schur complement based KKT preconditioners have also been developed for optimal control problems in cases where there are additional box constraints on the control and state variables \cite{PearsonStollWathen14} ($L^2$, $R\approx I$, $T\approx -I$). A general framework for using the Schur complement for the adjoint variable to precondition optimal control problems with box constraints is analyzed in an abstract function space setting in \cite{SchielaUlbrich14}, with only minimal assumptions on the operators $B$, $R$, and $T$. However, the specific Schur complement preconditioners presented in \cite{SchielaUlbrich14} are not regularization robust.

Certain non block diagonal approximations to KKT operators 
\cite{BankWelfertYserentant90} have been used to precondition elliptic
PDE constrained optimal control problems with $L^2$ regularization and observations 
\cite{HerzogSachs10, SchoberlZulehner07} ($L^2$, $B\approx I$, $R\approx I$, $T\approx -I$). Preconditioners of this type
have also been shown to be Hermitian positive definite in certain
nonstandard inner products, allowing the use of conjugate gradient as
a Krylov solver \cite{BramblePasciak88a, StollWathen08}.

Inner-outer methods where the Schur complement solve is performed
(exactly or approximately) with an additional inner stationary
iteration have also been proposed for several problems. These
include optimal Stokes control with $L^2$ regularization and observations \cite{ReesWathen11} ($L^2$, $B\approx I$, $R\approx I$, $T\approx -I$), and
optimal transport with a problem-specific diagonal regularization
operator \cite{BenziHaberTatalli11}. Recently, a method of this type
was proposed for optimal control problems with elliptic and parabolic
PDE constraints and smoothing regularization ($L^2$, $R^*R \approx \Delta + I$)
\cite{BarkerReesStoll16}. Regularization robustness was demonstrated
for the case $B =I$. 

\subsubsection{Block scaling}
\label{sec:block_scaling}

An abstract framework for constructing parameter independent (e.g., regularization robust) block diagonal preconditioners for saddle point systems is studied in \cite{Zulehner11} and applied to optimal control problems with elliptic and Stokes PDE constraints, with $B\approx I$, $R\approx I$, $T\approx -I$. In \cite{NielsenMardal10}, a certain class of block diagonal
KKT preconditioners for inverse problems (satisfying many assumptions)
was shown to be mesh independent and only weakly dependent on the
regularization parameter.\footnote{Note that in
  several of the papers cited in this subsection, the meaning of $B$
  and $T$ are switched relative to their use here.} One of
the central assumptions of the theory for this block diagonal
preconditioner is that the spectrum of the observation operator decays
exponentially. In a subsequent paper this assumption was replaced with
the similar assumption that the spectrum of the un-regularized KKT
system decays exponentially \cite{NielsenMardal13}.
Since the decay rates of these spectra depend on
the informativeness of the data, these assumptions are not applicable
(with good constants) in the context of inverse problems with highly
informative data. To overcome this limitation, recently the 
block diagonal preconditioner 
\begin{equation}
\label{eqn:mardal_nielsen_nordaas}
\begin{bmatrix}
\alpha I \\ & B^* B + \alpha \widehat{A^*A} \\ && \frac{1}{\alpha} I
\end{bmatrix}
\end{equation}
was proposed in \cite{MardalNielsenNordaas17}, 
where
$\widehat{A^*A}$ is a 4th order elliptic operator that is spectrally 
equivalent to $A^*A$. This preconditioner was proven to be mesh and
regularization robust for a specific source inversion problem with
$L^2$ regularization
($L^2$, $R\approx I$, $T\approx -I$). Despite substantial differences in motivation and analysis,
our proposed preconditioner \eqref{eqn:preconditioner} could be
considered as a generalization of this work to more general operators
$R$ and $T$. Specifically, setting $\rho=\alpha$ (instead of our
suggestion $\rho=\sqrt{\alpha}$), our preconditioner has the same
second and third diagonal blocks as the preconditioner
\eqref{eqn:mardal_nielsen_nordaas}, but contains a more elaborate 
operator depending on $R$ and $T$ in the first block.

\subsubsection{Multigrid}
\label{sec:multigrid}

Another family of KKT preconditioners for parameter estimation
problems are based on multigrid (see the review paper
\cite{BorziSchulz09} and references therein). These techniques are
classically categorized into three main categories: (1) speeding up or
preconditioning forward and adjoint solves, (2) using multigrid to
precondition the reduced Hessian, and (3) collective smoothing.

Methods in the first category do not use multigrid to address the
fundamental difficulties stemming from highly data informed inverse
problems: speeding up the forward (and adjoint) solves does not
address the challenge of creating a preconditioner that is data
scalable, because the number of forward/adjoint solves that must be
done scales with the informativeness of the data.

The primary difficulty with category (2) is that when the
regularization is chosen appropriately, the regularization and data
misfit terms of the reduced Hessian tend to ``fight'' each other (more
on this in Section \ref{sec:appropriate_regularization_full}). Thus smoothers for
the regularization term tend to be roughers for the data misfit term,
and vice versa. As a result, multigrid methods belonging to the second
category tend to be restricted to the case $R\approx I$.  We note in
particular the following papers \cite{AdavaniBiros08, AdavaniBiros10, AkcelikBirosDraganescuEtAl05, DruaguanescuDupont08, DruaguanescuSoane13}, on elliptic, parabolic, and Stokes
source inversion problems with this restriction.

In collective smoothing (3), one designs multigrid smoothers for the
entire KKT system (parameter, forward, and adjoint) at once
\cite{Borzi03, BorziGriesse05}. Collective smoothers also tend to
either require $R\approx I$, e.g., \cite{TakacsZulehner11}, or substantially
degrade in performance as the regularization parameter decreases,
e.g., \cite{AscherHaber03}.


\subsection{Commentary on solving the preconditioner subsystems}

Applying our preconditioner \eqref{eqn:preconditioner} requires the
solution of two subsystems with coefficient operators
\begin{equation}
\label{eqn:RR_TT}
\alpha R^*R + \rho T^*T
\end{equation} 
and 
\begin{equation}
\label{eqn:AA_BB}
B^*B + \rho A^*A,
\end{equation}
respectively. This can be a challenge.  However, reduced Hessian preconditioning and
KKT preconditioning for large scale inverse problems with highly informative data are fundamentally difficult endeavors, and the operators \eqref{eqn:RR_TT} and \eqref{eqn:AA_BB} have
many advantages over the alternatives. 

To begin with, we typically have easy access to the entries of the concrete matrix representations of these operators.\footnote{Although (dense) inverses of mass matrices can arise in concrete representations of these subsystems due to the adjoint operation, these inverse mass matrices can typically be replaced with spectrally equivalent sparse lumped mass approximations.} Thus we have at our disposal the entire arsenal of symmetric
positive definite sparse preconditioning techniques that deal with
matrix entries; e.g., incomplete factorizations, factorized sparse
approximate inverses \cite{Ferronato12}, and modern
multilevel techniques including algebraic multigrid and hierarchical
interpolative factorizations \cite{HoYing15}. This stands in direct
contrast to the reduced Hessian, which is dense owing to the inverses
of the forward and adjoint operators within it, and as such may be
accessed only via matrix-vector multiplies. 

Additionally, the data misfit Hessian (which often acts as a compact
operator) and the regularization operator (which often acts as a
differential operator) tend to act in opposition to each other by
construction.\footnote{By ``act in opposition,'' we mean that modes
  that are amplified by one operator tend to be diminished by the
  other operator, and vice versa. This is discussed more in Section
  \ref{sec:appropriate_regularization_full}.} Since the reduced
Hessian is the sum of these operators, it is difficult to design
preconditioners that are effective for both terms in the reduced
Hessian at the same time. In contrast, the different terms in our
subsystems tend not to act in opposition to each other.

In typical applications $R^*R$ is chosen to be an elliptic
differential operator, and $T$ is either identity-like, or acts
like a differential operator. Thus there is good reason to 
believe that multilevel techniques will be effective on the system
$\alpha R^*R + \rho T^*T$ in situations of practical interest. A
similar argument applies to $B^*B + \rho A^*A$ whenever the forward
operator $A$ is amenable to multilevel techniques. In the numerical
results section (Section \ref{sec:numerical_results}), we see that for a source inversion problem with an
elliptic PDE constraint, replacing the two subsystem solves with a few
algebraic multigrid V-cycles results in nearly the same convergence
rate as performing the solves exactly.

Of course, the operators in our subsystems are squared,
and such squaring should always done with caution. 
However, subsystems involving squared operators 
are also present in state of the art
preconditioners that have been proposed in the literature
(see Sections \ref{sec:adjoint_schur} and \ref{sec:block_scaling}). In
particular, a matrix spectrally equivalent to $B^*B + \rho A^*A$ shows
up in the preconditioner proposed in \cite{MardalNielsenNordaas17}. 


\subsection{Setting and notation}
\label{sec:notation}

For the purposes of this paper we consider the case for which all
spaces are finite dimensional Hilbert spaces, as might arise in stable
discretize-then-optimize methods \cite{Gunzburger03} for infinite
dimensional problems. To fix ideas, consider the case of an infinite
dimensional function space $\mathcal{U}_\infty$ approximated by a
finite dimensional function space $\mathcal{U}$, the elements of which
are in turn represented on a computer by lists of degrees of freedom
in $\mathbb{R}^n$ corresponding to a potentially non-orthogonal basis
$\basisoperator:\mathbb{R}^n \rightarrow \mathcal{U}$. Schematically,
\begin{equation*}
\underset{\substack{\infty-\text{dimensional} \\ \text{function space}}}{\mathcal{U}_\infty} \approx \underset{\substack{n-\text{dimensional} \\ \text{function space}}}{\mathcal{U}} \underset{\basisoperator}{\overset{\basisoperator^{-1}}{\rightleftharpoons}} \underset{\substack{\text{representation} \\ \text{space}}}{\mathbb{R}^n}.
\end{equation*}
Here we work in intermediate finite dimensional function spaces like
$\mathcal{U}$. In a representation space associated with a particular
non-orthogonal basis, all formulas from this paper remain essentially
the same, except linear operators are replaced with matrix
representations (arrays of numbers), abstract vectors are replaced
with their concrete representations (lists of numbers), and Gram
matrices (mass matrices) and their inverses appear in various
locations to account for the Riesz representation theorem for adjoints
in a non-orthogonal basis.

The parameter $q$, state $u$, adjoint $\adjoint$, and observations $y$
are assumed to reside in finite dimensional Hilbert spaces
$\mathcal{Q}$, $\mathcal{U}$, $\mathcal{V}$, and $\mathcal{Y}$ with
dimensions $n_q$, $n_u$, $n_u$, and $\nobs$ respectively. 
Linear operators, e.g., $A:\mathcal{U}\rightarrow \mathcal{V}$, are viewed as
abstract mappings between vector spaces, without reference to any
particular basis, except in the case where the domain and/or range are
of the form $\mathbb{R}^n$. Although we work with operators, we make
routine use of standard results for matrices that are easily extended
to the finite dimensional linear operator case, such as the existence
and properties of eigenvalues of certain classes of operators, and the
existence of the singular value decomposition. Transferring these
results from the matrix setting to the finite dimensional linear
operator setting is a straightforward process that involves working
with the matrix representations of the operators in bases that are
orthonormal with respect to the inner products on their domains and
ranges.\footnote{Note that such matrix representations with respect to
  orthonormal bases are generally not the same as the matrix
  representations that arise computationally within, say, a finite
  element method.} Concatenation of linear operators such as $BA$
denotes composition of linear maps, and concatenation of a linear
operator with a vector, as in $Au$, denotes the action of the operator
on the vector. Adjoints of operators are denoted by
superscript stars, as in $A^*$. Superscript stars on a vector denote the linear functional that takes inner products with that vector. Namely, $u^*: v \mapsto (u,v)$, where $(\cdot,\cdot)$ is the inner product for the space $u$ resides in.  Functions of a linear operator
such as inverses and square roots (where defined) are denoted in the
standard way, i.e., $A^{-1}, A^{1/2}$. Unless otherwise noted, the
norm of a vector, e.g., $\nor{u}$, is the norm associated with the
Hilbert space the vector resides in, and the norm of an operator,
e.g., $\nor{A}$, is the induced norm associated with the norms on the
domain and range spaces of the operator. Block operators, such as
\begin{equation*}
\begin{bmatrix}
X & Y \\ Z & W
\end{bmatrix} : \text{domain}(X)\oplus\text{domain}(Y) \rightarrow \text{range}(X)\oplus\text{range}(Z)
\end{equation*}
are defined by the blockwise action of their constituent operators, in
the usual way, and with the expected consistency restrictions on the
domains and ranges of the various blocks. Empty blocks are assumed to
contain the zero operator with the appropriate domain and range. We
use the notation $\Lambda = \text{diag}(\lambda_\ind)_{n,m}$ to denote
the linear map $\Lambda:\mathbb{R}^m\rightarrow \mathbb{R}^n$ whose
matrix representation in the standard basis is diagonal, with
$\ind$th diagonal entry $\lambda_\ind$. Likewise, when we write $\Phi
= \begin{bmatrix}\phi_1 & \phi_2 & \dots & \phi_m\end{bmatrix}$ for an
  operator $\Phi:\mathbb{R}^m\rightarrow \mathcal{X}$ and vectors
  $\phi_\ind \in \mathcal{X}$, we mean that $\phi_\ind$ is the result
  of applying $\Phi$ to the $\ind$th standard basis vector in
  $\mathbb{R}^{k}$ ($\phi_\ind$ is the ``$\ind$th column'' of
  $\Phi$). An operator is said to be square if the dimension of the
  domain and range are equal, and rectangular if the dimensions of the
  domain and range might differ.

The maximum and minimum singular values of an operator $Y$ are denoted
$\sigma_\text{max}(Y)$ and $\sigma_\text{min}(Y)$,
respectively. Similarly, the maximum and minimum eigenvalues of an
operator $X$ with strictly real eigenvalues are denoted
$\lambda_\text{max}(X)$ and $\lambda_\text{min}(X)$, respectively. The
condition number of an operator $X$ is denoted $\text{cond}(X)$.

\section{Derivation of the preconditioner}

\label{sec:augmented_background}

The preconditioner in \eqref{eqn:preconditioner} is derived from a
block diagonal approximation to the KKT operator associated with an
augmented Lagrangian formulation of optimization problem
\eqref{eqn:linear_optimization_problem}. In the following derivation,
it will be convenient to group the parameter and state variables into
a single vector $x:=\begin{bmatrix}q \\ u\end{bmatrix}$. With this
grouping, optimization problem \eqref{eqn:linear_optimization_problem}
takes the following standard quadratic programming form,
\begin{equation}
\label{eqn:generic_optimization_problem}
\begin{aligned}
\min_{x} &\quad \frac{1}{2}x^* M x - g^*x\\
\text{such that} &\quad Cx=f,
\end{aligned}
\end{equation}
where $g:=\begin{bmatrix}0 \\
B^*y\end{bmatrix}$, $C:=\begin{bmatrix}T & A\end{bmatrix}$, and $M$ is the (generally singular) operator,
\begin{equation*}
M := \begin{bmatrix}\alpha R^*R \\ & B^*B\end{bmatrix}.
\end{equation*}
The KKT operator from equation \eqref{eqn:KKT_matrix_full} then becomes,
\begin{equation}
\label{eqn:KKT_matrix_3x3_2x2}
K := \begin{bmatrix}\alpha R^*R & & T^* \\
& B^*B & A^* \\
T & A
\end{bmatrix} = \begin{bmatrix}
M & C^* \\ C
\end{bmatrix}.
\end{equation}
For non-singular $M$, it is well-established \cite{MurphyGolubWathen00} that the following positive definite block diagonal preconditioner,
\begin{equation}
\label{eqn:basic_schur_complement_preconditioner}
\begin{bmatrix}
M \\ & C M^{-1} C^*
\end{bmatrix},
\end{equation}
clusters the eigenvalues of the preconditioned operator onto at most
three distinct values. Note that the positive operator $C M^{-1} C^*$
is the negative Schur complement for the adjoint variable. Since the
objective block $M$ is singular whenever $B$ is not full rank (i.e.,
in the case of limited observations), we cannot directly use this
result. However, \eqref{eqn:generic_optimization_problem} has the same
solution as the following augmented optimization problem, 
\begin{equation*}
\label{eqn:generic_augmented_optimization_problem}
\begin{aligned}
\min_{x} &\quad \frac{1}{2}x^* M x - g^*x + \frac{\rho}{2} \nor{Cx-f}^2\\
\text{such that} &\quad Cx=f,
\end{aligned}
\end{equation*}
where the constraint is enforced strictly, but an additional quadratic
penalty term is added to the objective function to further penalize
constraint violations when an iterate is away from the optimal
point. The KKT operator for this augmented optimization problem is 
\begin{equation}
\label{eqn:augmented_kkt}
\begin{bmatrix}
M + \rho C^*C & C^* \\ C
\end{bmatrix}.
\end{equation}
With this augmentation, the objective block is now nonsingular
provided that $M$ is coercive on the null space of $C$
(i.e., the optimization problem is
well-posed). 

The positive definite block diagonal preconditioner analogous to
\eqref{eqn:basic_schur_complement_preconditioner} but based on the
augmented KKT operator \eqref{eqn:augmented_kkt} is 
\begin{equation}
\label{eqn:general_augmented_preconditioner}
\begin{bmatrix}
M + \rho C^*C \\ & C(M + \rho C^*C)^{-1}C^*
\end{bmatrix}.
\end{equation}
This preconditioner clusters the spectrum of the original
(non-augmented) KKT operator onto the union of two well-conditioned
intervals \cite{GolubGreifVarah06}. However, this preconditioner is
not practical since it is computationally difficult to perform solves
$(M + \rho C^*C)^{-1}$, as well as apply the Schur complement $C(M +
\rho C^*C)^{-1}C^*$ and its inverse. Thus we construct the
preconditioner in \eqref{eqn:preconditioner} by replacing these blocks
with cheaper approximations. 

Intuitively, when $\rho$ is large, constraint violations are more
strongly penalized by the objective, so the adjoint variable does not
need to ``work as hard'' to enforce the constraint. This manifests in
better conditioning of the Schur complement for the adjoint, $C(M +
\rho C^*C)^{-1}C^*$. Indeed, it is easy to see that $C(M + \rho
C^*C)^{-1}C^* \rightarrow \frac{1}{\rho}I$ as $\rho \rightarrow
\infty$. To this end, we expect the approximate
preconditioner
\begin{equation}
\label{eqn:penalty_preconditioner}
\begin{bmatrix}
M + \rho C^*C \\ & \frac{1}{\rho}I
\end{bmatrix},
\end{equation}
to perform well when $\rho$ is large. The preconditioner
\eqref{eqn:penalty_preconditioner} is, essentially, a mechanism for
using an unconstrained penalty method to precondition a constrained
optimization problem.

The augmented objective block, $M + \rho C^*C$, takes the form
\begin{equation*}
M + \rho C^*C = \begin{bmatrix}
\alpha R^*R + \rho T^*T & \rho T^*A \\ \rho A^*T & B^*B + \rho A^*A
\end{bmatrix}.
\end{equation*}
Since this $2 \times 2$ block operator is difficult to solve, we
cannot use preconditioner \eqref{eqn:penalty_preconditioner} directly,
and must make further approximations. In particular, the off-diagonal
blocks are scaled by $\rho$, so when $\rho$ is small we expect the
relative importance of these blocks to be reduced. Dropping the
off-diagonal blocks in $M + \rho C^*C$ and then substituting the
result into \eqref{eqn:penalty_preconditioner} yields our overall $3
\times 3$ block diagonal preconditioner \eqref{eqn:preconditioner},
\begin{equation*}
P:=\begin{bmatrix}
\alpha R^*R + \rho T^*T \\
& B^*B + \rho A^*A \\
& & \frac{1}{\rho} I
\end{bmatrix}.
\end{equation*}

One hopes that it is possible to choose $\rho$ large enough that the Schur complement is well approximated by $\frac{1}{\rho}I$, but at the same time small enough that the objective block is well-preconditioned by the block diagonal approximation. Our theory and numerical results in subsequent sections suggest that these competing interests can be balanced by choosing $\rho=\sqrt{\alpha}$, provided that the inverse problem is appropriately regularized. In the next section we provide an abstract theoretical analysis of the preconditioner without making any assumptions about the value of $\rho$. A more specific analysis for source inversion problems with spectral filtering regularization, which motivates our choice of $\rho$, is performed in Section \ref{sec:source_filter}.

\section{Abstract analysis of the preconditioner}
\label{sec:theory}

In this section we analyze the preconditioned KKT operator, showing
that it is well-conditioned if bounds on the arithmetic and geometric
means of certain damped projectors are satisfied. First, we highlight
the structure of the preconditioned KKT operator, state the necessary
arithmetic and geometric mean bounds, and recall a prerequisite result
from Brezzi theory. Then we prove bounds on the condition number of
the preconditioned KKT operator based on the arithmetic and geometric
mean bounds.

\subsection{Prerequisites}

\subsubsection{Preconditioned KKT operator}

Let $E$ denote the symmetrically preconditioned KKT operator,
\begin{equation*}
E := P^{-1/2}KP^{-1/2},
\end{equation*}
with $P$ and $K$ defined in \eqref{eqn:preconditioner} and
\eqref{eqn:KKT_matrix_3x3_2x2}, respectively. Direct calculation shows
that the symmetrically preconditioned KKT operator has the following
block structure,
\begin{equation}
\label{eqn:blockwise_E}
E = \begin{bmatrix}I - F^*F & & F^* \\ & I - G^*G & G^* \\ F & G\end{bmatrix},
\end{equation}
where the operators $F$ and $G$ are defined as
\begin{equation*}
F := T\left(\frac{\alpha}{\rho}R^*R + T^*T\right)^{-1/2}, \quad
G := A\left(\frac{1}{\rho}B^*B + A^*A\right)^{-1/2}.
\end{equation*}
For convenience, we further denote the objective and constraint blocks
of the preconditioned system by $X$ and $Y$, respectively, where
\begin{equation}
\label{eqn:defn_of_X_and_Y}
X := \begin{bmatrix}I - F^*F \\ & I - G^*G\end{bmatrix}, \quad Y := \begin{bmatrix}F & G\end{bmatrix},
\end{equation}
so that the preconditioned KKT operator takes the form
\begin{equation}
\label{eqn:E_XY}
E = \begin{bmatrix}
X & Y^* \\ Y
\end{bmatrix}.
\end{equation}

\subsubsection{Arithmetic and geometric mean assumptions}
\label{sec:am_gm_conditions}
The quality of the preconditioner depends on the arithmetic and geometric means of the following two \textit{damped projectors},\footnote{Recall that $X(\gamma I + X^*X)^{-1}X^*$ approximates the orthogonal projector onto the column space of $X$ for small $\gamma$. With this in mind, one can view an operator of the form $X(Y^*Y + X^*X)^{-1}X^*$ as an approximate projector onto the column space of $X$, damped by the operator $Y$. We call such operators \textit{damped projectors}.}
\begin{equation}
\label{eqn:QR}
Q_R := FF^* = T\left(\frac{\alpha}{\rho}R^*R + T^*T\right)^{-1}T^*,
\end{equation}
and
\begin{equation*}
Q_J := GG^* = A\left(\frac{1}{\rho}B^*B + A^*A\right)^{-1}A^*.
\end{equation*}
Note that if $T$ is invertible, we have
\begin{equation}
\label{eqn:QJ_J}
Q_J = T\left(\frac{1}{\rho}J^*J + T^*T\right)^{-1}T^*,
\end{equation}
where
\begin{equation}
\label{eqn:parameter_to_observable_map}
J := -B A^{-1} T
\end{equation}
is the \emph{parameter-to-observable map} that transforms candidate parameter fields into predicted observations. 

As damped projectors, it is easy to show that the eigenvalues of $Q_R$
and $Q_J$ are bounded between $0$ and $1$.  The degree to which the
eigenvalues of $Q_R$ are damped below $1$ is controlled by the
strength of the damping term $\frac{\alpha}{\rho}R^*R$ and its
interaction with the eigenstructure of $T$. Similarly, the degree of
damping of the eigenvalues of $Q_J$ is controlled by the strength of
the damping term $\frac{1}{\rho}J^*J$ and its interaction with the
eigenstructure of $T$ (or the interaction of the damping term
$\frac{1}{\rho}B^*B$ with the eigenstructure of $A$, when $T$ is not
invertible). 

\begin{assume}[Damped projector AM-GM bounds]
\label{assume:amgm}
We assume there exist constants $\beta, \delta$ such that the following bounds on the spectrum of the arithmetic and geometric means of the damped projectors hold:
\begin{enumerate}[label=\alph*), topsep=1em, itemsep=1em]
\item \quad 
$ \displaystyle 0 < \delta \le \frac{1}{2}\lambda_\text{min}\left(Q_R + Q_J\right),$
\item \quad $\lambda_\text{max}\left(Q_R Q_J\right)^{1/2} \le \beta < 1.$
\end{enumerate}
\end{assume}
Theorem \ref{thm:cond_bound} will establish that the larger $\delta$ is and the smaller $\beta$ is, the more effective preconditioner \eqref{eqn:preconditioner} is.

Qualitatively, if $T$ is invertible and the regularization is chosen
to act in opposition to the data misfit, as desired for the problem to
be properly regularized based on the analysis that will be performed
in Section \ref{sec:appropriate_regularization_full}, then $R$ will
act strongly on vectors that $J$ acts weakly on, and vice versa. Thus
we expect the damping in $Q_R$ to be strong where the damping in $Q_J$
is weak, and vice versa. Consequently, it is reasonable to hypothesize
that Assumption \ref{assume:amgm} will be satisfied with good
constants for inverse problems that are properly regularized. Making
this intuition precise requires careful analysis of the interaction
between the eigenstructures of $R$, $J$, and $T$, which must be done
on a case-by-case basis. We perform this analysis for the special case
of source inversion problems with spectral filtering regularization in
Section \ref{sec:source_filter}, and expect similar behavior to hold
in more general situations. 

\subsubsection{Brezzi theory for well posedness of saddle point systems}
\label{sec:coercive_saddle_point}
The proof of the coercivity bound for our preconditioned KKT operator invokes Brezzi theory for saddle point systems \cite{BrezziFortin91, Demkowicz06a, XuZikatanov00}. In particular, we use a recently discovered bound in \cite{Krendl13}, which is slightly sharper than bounds derived from the classical theory. Here we state the prerequisite theorem (without proof), and refer the reader to \cite{Krendl13} for more details. This theory can be stated in much greater generality than what we present here.

\begin{thm}[Krendl, Simoncini, and Zulehner]
\label{thm:generic_saddle_point}
Let $E$ be the saddle point system
\begin{equation*}
E = \begin{bmatrix}
X & Y^* \\ Y
\end{bmatrix},
\end{equation*}
where $X:\mathcal{X} \rightarrow \mathcal{X}$ is self-adjoint and positive semidefinite, and $Y:\mathcal{X} \rightarrow \mathcal{Y}$. Further suppose that 
\begin{itemize}
\item $X$ is coercive on the kernel of $Y$, i.e.,
\begin{equation*}
0 < a \le \inf_{\substack{\xpre \in \text{Ker}(Y) \\ \xpre \neq 0}} \frac{\xpre^*X\xpre}{\nor{\xpre}^2}.
\end{equation*}
\item $X$ is bounded, i.e., $\nor{X} < b$.
\item The singular values of $Y$ are bounded from below, i.e.,
\begin{equation*}
0 < c \le \sigma_\text{min}(Y).
\end{equation*}
\end{itemize}
Then the minimum singular value of $E$ is bounded from below, with the bound
\begin{equation}
\label{eqn:brezzi_improved}
\frac{a}{1 + \left(\frac{b}{c}\right)^2} \le \sigma_\text{min}(E).
\end{equation}
\end{thm}

\subsection{Bound on the condition number of the preconditioned KKT operator}
\label{sec:cond_bound}

To apply Brezzi theory (Theorem \ref{thm:generic_saddle_point}) to our problem, we need a coercivity bound for $X$ on the kernel of $Y$, a continuity bound for $X$ on the whole space, and a coercivity bound on $Y$, where the constants for these bounds are denoted $a$, $b$, and $c$, respectively. We use the particular structure of the KKT operator \eqref{eqn:KKT_matrix_3x3_2x2}, along with Assumption \ref{assume:amgm}, to derive these bounds in Section \ref{sec:M_rho_spectrum}. In Proposition \ref{prop:11_block_bounds} we derive bounds for $a$ and $b$, and then in Proposition \ref{prop:constraint_bounds} we derive a bound for $c$. 

In Section \ref{sec:kkt_upper_lower_cond} we derive well posedness and continuity bounds on the preconditioned KKT operator, $E$, and then combine these bounds to provide an upper bound on the condition number of $E$. Well posedness of $E$ is proven in Proposition \ref{prop:kkt_lower_bound}, using Brezzi theory in the form of Theorem \ref{thm:generic_saddle_point}. Continuity of $E$ is proven directly in Proposition \ref{prop:kkt_upper_bound}. Finally, the overall condition number bound for $E$ is given in Theorem \ref{thm:cond_bound}.

\subsubsection{Bounds on $X$ and $Y$}
\label{sec:M_rho_spectrum}

\begin{prop}[Bounds $a$, $b$ for $X$]
\label{prop:11_block_bounds}
The eigenvalues of $X$ restricted to the kernel of $Y$ are bounded below by $1-\beta$, where $\beta$ is defined in Assumption \ref{assume:amgm}. That is, 
\begin{equation*}
0 < 1-\beta \le \inf_{\substack{\xpre \in \text{Ker}(Y) \\ \xpre \neq 0}} \frac{\xpre^*X\xpre}{\nor{\xpre}^2}.
\end{equation*}
Additionally, 
\[
\nor{X} \le 1.
\]
\end{prop}

\begin{proof}
For vectors $\xpre \in \text{Ker}(Y)$, we have,
\begin{equation}
\label{eqn:augmented_eigenvalues}
\xpre^* X \xpre = \xpre^* (X + Y^* Y) \xpre \ge \lambda_\text{min}(X + Y^* Y) ||\xpre||^2.
\end{equation}
This augmented operator has the following block structure,
\begin{equation*}
X + Y^*Y = \begin{bmatrix}I -F^*F \\ & I - G^*G\end{bmatrix} + \begin{bmatrix}F^* \\ G^*\end{bmatrix}\begin{bmatrix}F & G\end{bmatrix} = \begin{bmatrix}
I & F^*G \\ G^*F & I
\end{bmatrix}.
\end{equation*}
Thus the eigenvalues $\lambda$ of $X + Y^*Y$ satisfy,
\begin{equation*}
\begin{bmatrix}I & F^*G \\ G^*F & I\end{bmatrix}\begin{bmatrix}\upre \\ \etapre\end{bmatrix}=\lambda\begin{bmatrix}\upre \\ \etapre\end{bmatrix},
\end{equation*}
or,
\begin{equation}
\label{eqn:offdiagonal_block_eigenvalue_system}
\begin{bmatrix} & F^*G \\ G^*F & \end{bmatrix}\begin{bmatrix}\upre \\ \etapre\end{bmatrix}=(\lambda-1)\begin{bmatrix}\upre \\ \etapre\end{bmatrix}.
\end{equation}
Solving for $u$ from the block equation associated with the first row block of \eqref{eqn:offdiagonal_block_eigenvalue_system} and substituting into the second yields,
\begin{equation*}
G^*FF^*G \etapre = (\lambda-1)^2 \etapre.
\end{equation*}
Thus, the magnitudes of the shifted eigenvalues, $|\lambda-1|$, are the square roots of the eigenvalues of $G^*FF^*G$. By a similarity transform, the eigenvalues of $G^*FF^*G$ are the same as the eigenvalues of the operator $FF^*GG^*$, and by the second part of Assumption \ref{assume:amgm}, we know that these eigenvalues are bounded above by $\beta$. Thus,
\begin{equation*}
|\lambda-1| \le \lambda_\text{max}(FF^*GG^*)^{1/2} \le \beta.
\end{equation*}
which implies,
\begin{equation*}
1 -\beta \le \lambda,
\end{equation*}
so that,
\begin{equation*}
\xpre^* X \xpre \ge (1 - \beta) \nor{\xpre}^2,
\end{equation*}
from which the inf-sup bound directly follows.

Since $FF^*$ and $GG^*$ are damped projectors, their eigenvalues reside in the interval $[0,1]$, as do the eigenvalues of $F^*F$ and $G^*G$. Using the definition of $X$ in \eqref{eqn:defn_of_X_and_Y}, this implies that the singular values of $X$ reside in the interval $[0,1]$, and so we have the upper bound $||X|| \le 1$.
\end{proof}

\begin{prop}[Bound $c$ for $Y$]
\label{prop:constraint_bounds}
The singular values of the preconditioned constraint are bounded below, with bound,
\begin{equation*}
0 < \sqrt{2 \delta} \le \sigma_\text{min}(Y).
\end{equation*}
\end{prop}

\begin{proof}
Since $G$ is invertible, $Y = \begin{bmatrix}F & G\end{bmatrix}$ has full row rank. Thus the singular values of $Y$ are the square roots of the eigenvalues of
\begin{equation*}
YY^* = FF^* + GG^*.
\end{equation*}
Recalling the arithmetic mean assumption (Assumption \ref{assume:amgm}a), we have
\begin{equation*}
0 < \delta \le \frac{1}{2}\lambda_\text{min}\left(FF^* + GG^*\right)
= \frac{1}{2}\lambda_\text{min}\left(YY^*\right),
\end{equation*}
or
\begin{equation*}
0 < \sqrt{2 \delta} \le \sigma_\text{min}(Y).
\end{equation*}
\end{proof}

\subsubsection{Well posedness, continuity, and conditioning of the preconditioned KKT operator, $E$}
\label{sec:kkt_upper_lower_cond}

\begin{prop}[Well posedness of $E$]
\label{prop:kkt_lower_bound}
The singular values of $E$ have the following lower bound:
\begin{equation*}
0 < \frac{2}{3} (1-\beta)\delta \le \sigma_\text{min}(E).
\end{equation*}
\end{prop}
\begin{proof}
Based on the results of Propositions \ref{prop:11_block_bounds} and \ref{prop:constraint_bounds}, and the block structure of $E$ from \eqref{eqn:E_XY}, we can apply bound \eqref{eqn:brezzi_improved} from Theorem \ref{thm:generic_saddle_point} to $E$ with $a = 1-\beta$, $b = 1$, and $c^2 = 2\delta$. Doing this and then using the fact that $0 < \delta \le 1$, we get the desired lower bound on the minimum singular value:
\begin{equation*}
\sigma_\text{min}(E) \ge \frac{1-\beta}{1 + \frac{1}{2 \delta}} = \frac{2 (1-\beta)\delta}{1 + 2 \delta} \ge \frac{2}{3} (1-\beta)\delta.
\end{equation*}
\end{proof}

\begin{prop}[Continuity of $E$]
\label{prop:kkt_upper_bound}
The singular values of $E$ are bounded above by $2$. I.e.,
\begin{equation*}
\sigma_\text{max}(E) \le 2.
\end{equation*}
\end{prop}

\begin{proof}
To prove the upper bound, we directly estimate the quantity $|w_1^*E w_2|$ for arbitrary $w_1,w_2$. Denote the blocks of $w_1$ and $w_2$ by,
\begin{equation*}
w_1  = \begin{bmatrix}\qpre_1 \\ \upre_1 \\ \etapre_1\end{bmatrix}, \quad w_2 = \begin{bmatrix}\qpre_2 \\ \upre_2 \\ \etapre_2\end{bmatrix}.
\end{equation*}
Recalling the blockwise definition of $E$ from \eqref{eqn:blockwise_E} and using the triangle inequality, we have
\begin{align}
|w_1^*E w_2| &= \left\lvert\begin{bmatrix}\qpre_1^* & \upre_1^* & \etapre_1^*\end{bmatrix} \begin{bmatrix}I - F^*F & & F^* \\ & I - G^*G & G^* \\ F & G\end{bmatrix}\begin{bmatrix}\qpre_2 \\ \upre_2 \\ \etapre_2\end{bmatrix} \right\rvert \nonumber \\
&= |\qpre_1^*(I-F^*F)\qpre_2 + \qpre_1^*F^*\etapre_2 +\upre_1^*(I-G^*G)\upre_2 + \upre_1^*G^*\etapre_2 + \etapre_1^*F\qpre_2 + \etapre_1^*G\upre_2| \nonumber \\
&\le |\qpre_1^*(I-F^*F)\qpre_2| + |\qpre_1^*F^*\etapre_2| + |\upre_1^*(I-G^*G)\upre_2| + |\upre_1^*G^*\etapre_2| + |\etapre_1^*F\qpre_2| + |\etapre_1^*G\upre_2|. \label{eqn:zEw0}
\end{align}
Since the operators $F$ and $G$ have singular values between zero and one, we can eliminate all of the intermediate operators in \eqref{eqn:zEw0}, yielding
\begin{equation}
\label{eqn:zEw1}
|w_1^*E w_2| \le  \nor{\qpre_1} \nor{\qpre_2} + \nor{\qpre_1} \nor{\etapre_2} + \nor{\upre_1} \nor{\upre_2} + \nor{\upre_1} \nor{\etapre_2} + \nor{\etapre_1} \nor{\qpre_2} + \nor{\etapre_1} \nor{\upre_2}.
\end{equation}
By Cauchy-Schwarz, three of the terms on the right hand side of \eqref{eqn:zEw1} can be estimated as follows:
\begin{align*}
\nor{\qpre_1} \nor{\qpre_2} + \nor{\upre_1} \nor{\etapre_2} + \nor{\etapre_1} \nor{\upre_2} &\le \left(\nor{\qpre_1}^2 + \nor{\upre_1}^2 + \nor{\etapre_1}^2\right)^{1/2}\left(\nor{\qpre_2}^2 + \nor{\upre_2}^2 + \nor{\etapre_2}^2\right)^{1/2} \\
&= \nor{w_1} \nor{w_2}.
\end{align*}
The other three terms can be estimated similarly:
\begin{equation*}
\nor{\qpre_1} \nor{\etapre_2} + \nor{\upre_1} \nor{\upre_2} + \nor{\etapre_1} \nor{\qpre_2} \le \nor{w_1} \nor{w_2}.
\end{equation*}
Thus we have the overall estimate
\begin{equation*}
|w_1^*E w_2| \le 2 \nor{w_1} \nor{w_2},
\end{equation*}
which implies $\sigma_\text{max}(E) \le 2$, as required.
\end{proof}

\begin{thm}[Conditioning of $E$]
\label{thm:cond_bound}
\begin{equation*}
\cond{E} \le \frac{3}{(1-\beta)\delta}.
\end{equation*}
\end{thm}
\begin{proof}
Divide the upper bound from Proposition \ref{prop:kkt_upper_bound} by the lower bound from Proposition \ref{prop:kkt_lower_bound}.
\end{proof}

\section{Spectral filtering and appropriate regularization assumptions}
\label{sec:appropriate_regularization_full}

To better characterize the constants $\delta$ and $\beta$ in the
condition number bound in Theorem \ref{thm:cond_bound}, in this
section we propose \emph{appropriate regularization assumptions}
({\regassumption}) that limit the degree to which the inverse problem
can be over- or under- regularized. These assumptions are motivated by
an analysis of the error in the reconstruction of the parameter
(Sections \ref{sec:reduced_problem} and
\ref{sec:spectral_filtering_defn}), and apply to spectral filtering
regularization operators (Definition \ref{def:spectral_filter}). Since
one part of {\regassumption} (specifically, {\regassumption}b) is
novel, we discuss that part in greater detail. Much of the
development we present leading up to (but not including)
{\regassumption} mirrors the classical treatment presented in
\cite{EnglHankeNeubauer96}.

Since construction of spectral filtering regularization operators is too expensive for large scale inverse problems with highly informative data, {\regassumption} is used for theoretical analysis only. In Section \ref{sec:source_filter} we will prove that satisfying {\regassumption} implies the existence of good constants $\delta$ and $\beta$ for source inversion problems, thereby guaranteeing that our preconditioner will perform well on these problems. 

\subsection{The reduced problem and decomposition of error}
\label{sec:reduced_problem}
Although we take a full space approach for solving optimization
problem \eqref{eqn:linear_optimization_problem}, for the purpose of
analysis it is useful to consider the reduced version of the problem
in which the constraint is eliminated by viewing the state $u$ as
an implicit function of the parameter $q$ via solution of the state
equation. This yields the following unconstrained optimization problem
in $q$ only:
\begin{equation}
\label{eqn:reduced_problem}
\min_q \quad \frac{1}{2}\nor{Jq - y}^2 + \frac{\alpha}{2}\nor{Rq}^2,
\end{equation}
where we recall from \eqref{eqn:parameter_to_observable_map} that the
\emph{parameter-to-observable map} $J$ is defined as $J:=
-BA^{-1}T$. The solution $q$ to this reduced problem is the solution
to the normal equations, 
\begin{equation}
\label{eqn:normal_equations}
Hq = J^* y,
\end{equation}
where
\begin{equation}
\label{eqn:reduced_hessian}
H:=J^*J + \alpha R^*R
\end{equation}
is the Hessian of the reduced optimization problem
\eqref{eqn:reduced_problem}, which we call the \emph{reduced
  Hessian}. 
The reduced Hessian has been the target of much of the previous work
on preconditioners for inverse problems (see Section
\ref{sec:reduced_hessian_literature}), including the method we
numerically compare our preconditioner to in Section
\ref{sec:numerical_results}.

From an optimization perspective, the purpose of the regularization is
to make optimization problem \eqref{eqn:reduced_problem} well-posed by
introducing curvature in the objective function in directions that are
in the (numerical) null space of $J$. However, in the context of
inverse problems the regularization is primarily seen as a means of
stabilizing the inversion with respect to noise in the observations.

Recall from \eqref{eqn:obs_with_noise} that the observations we use
for the inversion are corrupted by additive noise $\noisevector$ via
the formula 
\begin{equation}
\label{eqn:reduced_obs}
y = y_\text{true} + \noisevector = J q_\text{true} + \noisevector,
\end{equation}
where $q_\text{true}$ is the unknown true parameter and $y_\text{true}=J q_\text{true}$ are the observations that would have been obtained if there were no noise. Substituting \eqref{eqn:reduced_obs} into \eqref{eqn:normal_equations} and then subtracting the result from $q_\text{true}$, we see that the error takes the form
\begin{equation*}
q_\text{true} - q = \errunder + \errover,
\end{equation*}
consisting of a term
\begin{equation}
\label{eqn:under_reg_error}
\errunder := - \left(J^*J + \alpha R^*R\right)^{-1}J^* \noisevector
\end{equation} 
that depends on the noise, and a term
\begin{equation}
\label{eqn:over_reg_error}
\errover := \left(I - \left(J^*J + \alpha R^*R\right)^{-1}J^*J\right) q_\text{true}
\end{equation} 
that does not. From the form of equations \eqref{eqn:under_reg_error}
and \eqref{eqn:over_reg_error}, a trade-off is evident: strengthening
the regularization tends to reduce $\errunder$ at the expense of
increasing $\errover$, and weakening the regularization tends to
reduce $\errover$ at the expense of increasing $\errunder$. To achieve a
good reconstruction of the parameter, it is desirable for both of
these terms to be as small in magnitude as possible. To investigate
this trade-off in more detail, we restrict our subsequent analysis to
the 
special case of spectral filtering regularization, which
we define and discuss in the following section. This will provide
convenient bases to diagonalize the operators $- \left(J^*J + \alpha
R^*R\right)^{-1}J^*$ and $\left(I - J^*J + \alpha
R^*R\right)^{-1}J^*J$, and hence allow us understand the errors
$\errunder$ and $\errover$ in a per-component manner.

\subsection{Spectral filtering regularization}
\label{sec:spectral_filtering_defn}

\begin{defn} 
\label{def:spectral_filter}
An operator $R$ is a spectral filtering regularization operator for a linear inverse problem with parameter-to-observable map $J$ if $R$ and $J$ share a common basis of right singular vectors. That is, there exist 
\begin{itemize}
\item unitary operators $U:\mathbb{R}^{\nobs}\rightarrow \mathcal{Y}$, $V:\mathbb{R}^{n_q}\rightarrow\mathcal{Q}$, and $\Phi:\mathbb{R}^{n_q} \rightarrow \mathcal{Q}$, and
\item non-negative diagonal operators $\Sigma_J = \text{diag}(d_\ind)_{\nobs,n_q}$, and $\Sigma_R = \text{diag}(r_\ind)_{n_q, n_q}$
\end{itemize} 
such that
\begin{equation}
\label{eqn:generalized_svd}
\begin{cases}
J = U \Sigma_J \Phi^*, \\
R = V \Sigma_R \Phi^*.
\end{cases}
\end{equation}
By convention we order the singular values $d_k$ of $J$ in descending
order ($d_k \ge d_{k+1}$). In the case where $\nobs < n_q$, for
convenience we define $d_k:=0$ for $k=\nobs+1,\dots,n_q$. Note that
the descending order for $d_k$ forces an order (possibly non-monotone)
for the singular values $r_k$ of $R$. We use $\parametermode_k$ to
denote the $k$th right singular vector shared by $J$ and $R$. That is,
$\Phi = \begin{bmatrix}\parametermode_1 & \parametermode_2 & \dots &
  \parametermode_{n_q}\end{bmatrix}$. 
\end{defn}

Spectral filtering regularization is ideally suited for inverse
problems---by manipulating the regularization singular values
$r_\ind$, one can selectively filter out undesirable components of the
parameter from the reconstruction without affecting the reconstruction
of the desirable components. The larger $r_\ind$, the more component
$\parametermode_\ind$ is penalized, and vice versa. Limiting cases of
spectral filtering regularization include:
\begin{itemize}
\item identity regularization ($R=I$), where all
singular vectors are penalized equally, and 
\item truncated SVD, where singular vectors $\parametermode_\ind$ are
  not penalized at all if $\parametereigenvalue_\ind$ is above a given
  threshold, but are penalized infinitely\footnote{That is, the
    reconstruction of the component of $q$ in the direction
    $\parametermode_\ind$ is set to zero.} otherwise.  
\end{itemize}
Spectral filtering regularization is routinely used for small to moderate sized inverse problems, and for large inverse problems that admit low-rank approximations to the parameter-to-observable map. However, aside from identity regularization, spectral filtering regularization is generally computationally infeasible for large-scale inverse problems with highly informative data. In fact, spectral filtering regularization requires computing the dominant singular vectors and singular values of $J$ in order to construct $R$, and the number of dominant singular vectors of $J$ scales with the informativeness of the data. Thus we view spectral filtering as an idealized form of regularization that practical regularization operators attempt to approximate. For a more
comprehensive discussion of spectral filtering and its relation to
other regularizations, we refer the reader to the
classic monograph \cite{EnglHankeNeubauer96}.

For spectral filtering regularization, we can formulate expressions
for the errors in the reconstruction on a per-component
manner. Substituting the singular value decomposition factors from
\eqref{eqn:generalized_svd} into the error
expressions from \eqref{eqn:under_reg_error} and
\eqref{eqn:over_reg_error}, and then performing some algebraic
manipulations, yields 
\begin{align}
\errunder &= -\Phi~ \text{diag}\left(\frac{d_\ind}{d_\ind^2 + \alpha r_\ind^2}\right) U^* \noisevector, \label{eqn:underreg_diagonalized} \\
\errover &= \Phi~ \text{diag}\left(\frac{\alpha r_\ind^2}{d_\ind^2 + \alpha r_\ind^2}\right) \Phi^* q_\text{true}. \label{eqn:overreg_diagonalized}
\end{align}
From \eqref{eqn:underreg_diagonalized}, we see that the regularization
should not be weak (small $\alpha r_\ind^2$) in directions
$\parametermode_\ind$ to which the observations are insensitive (small
$d_\ind^2$). Otherwise the noise associated with observations of those
directions will be highly amplified, leading to large errors. In such
a scenario we say that the problem is \emph{under-regularized}.

On the other hand, \eqref{eqn:overreg_diagonalized} shows that strong
regularization can also lead to large errors. In directions
$\parametermode_\ind$ for which observation data is lacking (or
dominated by noise), there is no hope to reconstruct the component of
the parameter in that direction, so some degree of error in $\errover$
is to be expected. However, if $d_\ind$ is large then the observations
are highly sensitive to changes to the parameter in direction
$\parametermode_\ind$, so it is likely that the observations
associated with direction $\parametermode_\ind$ contain more signal
than noise. That is, when $d_\ind$ is large, it is likely that the
component of the parameter $q_\text{true}$ in direction
$\parametermode_\ind$ can, in principle, be inferred from the
data. Hence, if the regularization is strong (large $\alpha r_\ind^2$)
in directions for which the parameter-to-observable map is also strong
(large $d_\ind^2$), the reconstruction will contain substantial
\emph{unnecessary} error due to the regularization. In this scenario
we say that the problem is \emph{over-regularized}. To simultaneously
avoid under- and over- regularization, the regularization should be
strong in directions where the parameter-to-observable map is weak,
and weak in directions where the parameter-to-observable map is
strong.

\subsection{Appropriate regularization assumptions}

In light of the preceding discussion of over- and under- regularization error and spectral filtering, we propose the following \emph{appropriate regularization assumptions}.

\begin{assume}[Appropriate regularization]
\label{condition:appropriate_regularization}
There exist constants $\reglower$ and $\regupper$  such that,
\begin{enumerate}[label=\alph*), topsep=1em, itemsep=1em]
\item \quad $ \displaystyle 0 < \reglower \le \parametereigenvalue_\ind^2 + \alpha r_\ind^2 $,
\item \quad $ \displaystyle \parametereigenvalue_\ind r_\ind \le \regupper < \infty $,
\end{enumerate}
for all $\ind=1,2,\dots,n_q$.
\end{assume}

{\regassumption}a is already required for linear optimization problem
\eqref{eqn:linear_optimization_problem} to be well-posed. It says that
the regularization cannot be arbitrarily small in basis directions
$\parametermode_\ind$ to which the observations are insensitive, but
allows the regularization to be arbitrarily small in directions
$\parametermode_\ind$ to which the observations are sensitive. In
contrast, {\regassumption}b prevents the regularization from being
large in basis directions $\parametermode_\ind$ to which the observations
are sensitive, but still allows the regularization singular values
to diverge ($r_\ind \rightarrow \infty$ as $\ind \rightarrow \infty$),
as long as the sensitivity of the observations to changes to the
parameter, $\parametereigenvalue_\ind$, goes to zero in the inverse
manner. Informally, {\regassumption}a says that the problem is not
under-regularized, and {\regassumption}b says that the problem is not
over-regularized. 

Since {\regassumption}a is standard, we do not discuss it further. The
motivation for {\regassumption}b is less obvious, so we provide a more
in-depth discussion of it. To begin with, the multiplicative nature of
{\regassumption}b makes it a relatively weak assumption compared to
other possible candidates for preventing over-regularization. In
particular, observe that the eigenvalues of the regularization
preconditioned reduced Hessian, $R^{-*} H R^{-1}$, are
$\parametereigenvalue_\ind^2/r_\ind^2 + \alpha$.  Thus situations 
in which the strength of the regularization operator on a mode is
inversely proportional to how informed that mode is (i.e., $r_\ind
\approx \frac{1}{\parametereigenvalue_\ind}$) can lead to arbitrarily
poor conditioning of the regularization preconditioned reduced Hessian
while still satisfying {\regassumption}b with a constant of order one.

An instructive model problem that illustrates {\regassumption}b is the
Poisson source inversion problem on a rectangular domain, with
Laplacian regularization, zero Dirichlet boundary conditions for both
$A$ and $R$, and distributed observations of the first $\nobs$ Fourier
modes of the state variable in the domain. That is,
\begin{itemize}
\item $T=I$,
\item  $A=R=\Delta_D$, where $\Delta_D$ is the Laplacian operator with zero Dirichlet boundary conditions, and 
\item $B:\mathcal{U} \rightarrow \mathbb{R}^\nobs$ is a wide rectangular operator with Fourier modes as right singular vectors (the same as $A$ and $R$), but with singular values $\sigma_\ind = 1$, $k=1,\dots,\nobs$. 
\end{itemize}
Recalling that $J=-BA^{-1}T=-B \Delta_D^{-1}$, we see that
\begin{equation*}
\parametereigenvalue_\ind = \begin{cases}\frac{1}{\lambda_\ind}, &\quad \ind=1,\dots,\nobs, \\
0, &\quad \ind > \nobs,
\end{cases}
\end{equation*}
where $\lambda_\ind$ is the $\ind$th eigenvalue of the Laplacian
$\Delta_D$. At the same time, the singular values of $R$ are $r_\ind =
\lambda_\ind$. Thus $\parametereigenvalue_\ind r_\ind = 1$ for
$\ind=1,\dots,\nobs$ and $\parametereigenvalue_\ind r_\ind=0$ for $k>\nobs$,
so {\regassumption}b holds with constant $\regupper = 1$, regardless
of the number of observations, $\nobs$.

\section{Analysis of the source inversion problem with spectral filtering regularization}
\label{sec:source_filter}

In Section \ref{sec:am_gm_conditions} we hypothesized that the damped
projector arithmetic and geometric mean assumptions (Assumption
\ref{assume:amgm}) are satisfied with good constants $\delta$ and
$\beta$ whenever an inverse problem is properly regularized. Then in Section \ref{sec:appropriate_regularization_full} we formulated another assumption ({\regassumption}) that quantifies the concept of proper regularization for spectral filtering regularization operators. Here we show that {\regassumption} implies Assumption \ref{assume:amgm} for the source inversion problem. Specifically, in Theorem \ref{thm:reg_to_amgm} and Corollary \ref{cor:reg_to_amgm} we prove quantitative bounds on the constants $\delta$ and $\beta$ for source inversion problems that are neither over- nor under- regularized in the manner made precise by {\regassumption}. The more appropriate to the problem the regularization is, the better the bounds.

\begin{defn} 
\label{def:source}
An inverse problem is a source inversion problem if the parameter $q$ being inverted for is the right-hand-side of the state equation. That is, $T=-I$, and state equation \eqref{eqn:state_equation} takes the form,
\begin{equation*}
Au = q.
\end{equation*}
\end{defn}

\begin{thm}
\label{thm:reg_to_amgm}
Let $R$ be a spectral filtering regularization operator for a source inversion problem (see Definitions \ref{def:spectral_filter} and \ref{def:source}). If $R$ satisfies appropriate regularization {\regassumption} with constants $\reglower$ and $\regupper$, then Assumption \ref{assume:amgm} is also satisfied, with constants
\begin{equation*}
\delta = \frac{1}{2}\left(1 + \frac{\alpha}{\rho^2}\regupper^2\right)^{-1} \quad \text{and} \quad
\beta = \left(1 + \frac{1}{\rho} \reglower \right)^{-1/2}.
\end{equation*}
\end{thm}

\begin{proof}
For $\delta$, we seek a lower bound on the eigenvalues of the
arithmetic mean of the damped projectors $Q_R$ and $Q_J$ (as defined
in \eqref{eqn:QR} and \eqref{eqn:QJ_J}, respectively), while for
$\beta$ we seek an upper bound on their geometric mean. For source
inversion problems these damped projectors take the form 
\begin{equation*}
Q_R = \left(\frac{\alpha}{\rho} R^*R + I\right)^{-1}  \quad \text{and} \quad
Q_J = \left(\frac{1}{\rho} J^*J + I\right)^{-1}.
\end{equation*}
Furthermore, for spectral filtering regularization, $R^*R$ and $J^*J$
share the same eigenvectors, and have eigenvalues $r_\ind^2$ and
$\parametereigenvalue_\ind^2$, respectively. Thus the eigenvalues
$\delta_\ind$ of the arithmetic mean $\tfrac{1}{2}(Q_R+Q_J)$ can be estimated as
\begin{equation*}
\delta_\ind = \frac{1}{2}\left(\frac{1}{\frac{\alpha}{\rho} r_\ind^2 + 1} + \frac{1}{\frac{1}{\rho} \parametereigenvalue_\ind^2 + 1}\right) 
\ge \frac{1}{2}\left(1 + \frac{\alpha}{\rho^2}\parametereigenvalue_\ind^2 r_\ind^2\right)^{-1} 
\ge \frac{1}{2}\left(1 + \frac{\alpha}{\rho^2}\regupper^2\right)^{-1}.
\end{equation*}
In the first inequality we have combined fractions, and used the non-negativity of $r_\ind^2, \parametereigenvalue_\ind^2$ and monotonicity of the function $f(x)=x/(a+x)$. In the second inequality we have used {\regassumption}b.

Similarly, we use the {\regassumption}a to bound the eigenvalues
$\beta_\ind$ of the geometric mean $(Q_R Q_J)^{1/2}$ as  
\begin{equation*}
\beta_\ind = \left(\frac{1}{\frac{\alpha}{\rho}r_\ind^2 + 1} \cdot \frac{1}{\frac{1}{\rho}\parametereigenvalue_\ind^2 + 1}\right)^{1/2}
\le \left(1 + \frac{\alpha}{\rho}r_\ind^2 + \frac{1}{\rho}\parametereigenvalue_\ind^2\right)^{-1/2}
\le \left(1 + \frac{1}{\rho}\reglower\right)^{-1/2}.
\end{equation*}
\end{proof}

The following corollary of Theorem \ref{thm:reg_to_amgm} shows that the preconditioner will be effective in the low to moderate regularization regime ($\alpha \le 1$) if we choose $\rho=\sqrt{\alpha}$.

\begin{cor}
\label{cor:reg_to_amgm}
If the conditions of Theorem \ref{thm:reg_to_amgm} are satisfied, and $\alpha \le 1$, and the regularization parameter is chosen as $\rho=\sqrt{\alpha}$, then Assumption \ref{assume:amgm} is satisfied, with constants 
\begin{equation*}
\delta = \frac{1}{2}\left(1 + \regupper^2\right)^{-1} \quad \text{and} \quad
\beta = \left(1 + \reglower\right)^{-1/2}.
\end{equation*}
\end{cor}

\begin{proof}
Substituting in $\rho = \sqrt{\alpha}$ into the results of Theorem
\ref{thm:reg_to_amgm}, we immediately have the desired lower bound on
the arithmetic mean of damped projectors with constant $\delta =
\frac{1}{2}\left(1 + \regupper^2\right)^{-1}$. For the geometric mean,
Theorem \ref{thm:reg_to_amgm} implies 
\begin{equation*}
\lambda_\text{max}\left(Q_R Q_J\right)^{1/2} \le \left(1 + \alpha^{-1/2} \reglower\right)^{-1/2}.
\end{equation*}
But note that for $\alpha \le 1$ we have
\begin{equation}
\label{eqn:drop_alpha_dependence}
\left(1 + \alpha^{-1/2}\reglower\right)^{-1/2} \le (1 + \reglower)^{-1/2},
\end{equation}
and so we get the desired upper bound with $\beta = (1 + \reglower)^{-1/2}$.
\end{proof}


\section{Numerical results}
\label{sec:numerical_results}

We apply our method to a Poisson source inversion problem with
pointwise observations randomly distributed throughout a rectangular
domain $\Omega=[0,1.45] \times [0,1]$, using Laplacian
regularization. Specifically, we take $q$, $u$, and $v$ to reside in
the space of continuous piecewise linear functions on a uniform
triangular mesh with mesh size parameter $h$, with the $L^2$ inner product. The state equation
\begin{equation*}
Au := \Delta_D u = q,
\end{equation*}
is the Poisson equation discretized by the finite element method, with
homogeneous Dirichlet boundary conditions enforced by the symmetric Nitsche
method \cite{Nitsche71}. Pointwise observations of the form
\begin{equation*}
y_\ind = u(x_\ind),
\end{equation*} 
are taken for a collection of points $\{x_\ind \in
\Omega\}_{\ind=1}^{\nobs}$, shown in Figure
\ref{fig:true_exact_obs}. Noise is not included in the inverse problem
since we are interested in preconditioners for the low noise, high
data, small regularization limit. The regularization operator is
defined by 
\begin{equation*}
R^*R := \Delta_N + t I,
\end{equation*}
where $\Delta_N$ is the Laplacian operator with Neumann boundary conditions discretized by the finite element method, and $t=1/10$. The combined operator $R^*R$ is used directly; in fact, the solution algorithm does not require $R$ explicitly.\footnote{Both $A$ and $R^*R$ should be viewed as finite dimensional discretizations of densely defined unbounded operators acting $L^2(\Omega)$.}

\begin{figure}[htbp]
\centering
\includegraphics[scale=0.5]{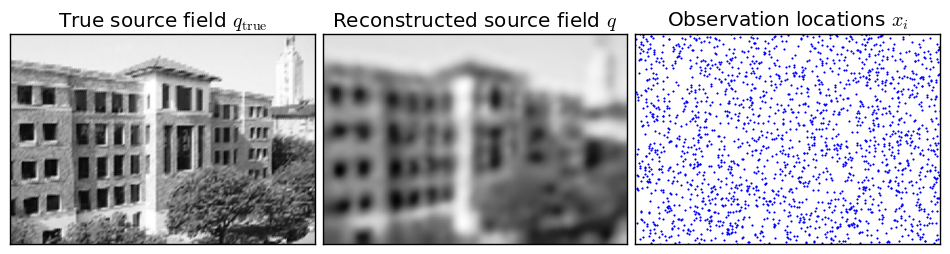}
\caption{Left: True source field $q_\text{true}$ used for all inversions. Center: Reconstruction $q$ for the case of $\nobs=2000$ observations with regularization parameter $\alpha=10^{-8}$ and mesh size $h=\sqrt{2} \cdot 10^{-2}$. Right: Observation locations $x_\ind$, denoted by dots.}
\label{fig:true_exact_obs}
\end{figure}

The true source field, $q_\text{true}$, used to generate the
observations, $y_\ind$, is a grayscale image of the Peter O'Donnell
Jr.\ building at the University of Texas at Austin, scaled to contain
values in $[0,1]$, and shown in Figure \ref{fig:true_exact_obs}. The
combination of sharp edges and smooth features in this image make this
an ideal test case for highly informative data and small
regularization. 

Abstract vectors $q,u,\adjoint$ are represented concretely by lists of
nodal degrees of freedom
$\concrete{q},\concrete{u},\concrete{\adjoint}$, respectively. The norm of a concrete vector, e.g., $\nor{\concrete{q}}$, is the Euclidean norm (square root of the sum of the squares of the entries). Since we use uniform meshes and present only relative errors, this is spectrally equivalent to using the function space $L^2$ norm on the underlying function being represented by the concrete vector. We use
the FEniCS \cite{LoggMardalGarth12} 
package to assemble
concrete matrix representations of $A$, $R^*R$, $T$, and $I$, which
are denoted $\concrete{A}$, $\concrete{R^*R}$, $\concrete{T}$, and
$\concrete{W}$, respectively. The diagonal lumped mass matrix is
denoted $\concrete{W}_L$, with diagonal entries given by row sums of
the mass matrix: $\left(\concrete{W}_L\right)_{ii} = \sum_j
\concrete{W}_{ij}$. The concrete sparse matrix representation of the
observation operator is denoted $\concrete{B}$. Its $(i,j)$ entry,
$\concrete{B}_{ij}$, equals the evaluation of the $j$th basis function
at the $i$th observation location. 

In a concrete basis, the KKT operator \eqref{eqn:KKT_matrix_full} becomes,
\begin{equation}
\label{eqn:discrete_kkt}
\begin{bmatrix}
\alpha \concrete{R^*R} & & -\concrete{W} \\ & \concrete{B}^T\concrete{B} & \concrete{A}^T \\ -\concrete{W} & \concrete{A}
\end{bmatrix}
\begin{bmatrix}
\concrete{q} \\ \concrete{u} \\ \concrete{\adjoint}
\end{bmatrix}
=
\begin{bmatrix}
0 \\ \concrete{B}^T y \\ 0
\end{bmatrix}.
\end{equation}
The reconstructed function $q$ based on the exact\footnote{By
  ``exact,'' we mean that the result of a computation is accurate to
  tolerance $10^{-12}$ or smaller.} solution of this KKT system with
regularization parameter $\alpha=10^{-8}$ is shown in Figure
\ref{fig:true_exact_obs}.  

In a concrete basis the preconditioner \eqref{eqn:preconditioner} becomes
\begin{equation}
\label{eqn:BDAL_exact}
\concrete{P} = \begin{bmatrix}
\alpha \concrete{R^*R} + \rho \concrete{W} \\
& \concrete{B}^T\concrete{B} + \rho \concrete{A}^T \concrete{W}^{-1} \concrete{A} \\
&& \frac{1}{\rho} \concrete{W}
\end{bmatrix}.
\end{equation}
In our numerical experiments, we consider three variants of this preconditioner.
\begin{itemize}
\item \textbf{BDAL, exact}: all solves in preconditioner \eqref{eqn:BDAL_exact} are performed exactly.
\item \textbf{BDAL, lumped mass, exact}: the mass matrix $\concrete{W}$ is replaced with the lumped mass matrix $\concrete{W}_L$, but preconditioner solves are performed exactly with this replacement.
\item \textbf{BDAL, lumped mass, multigrid}: the mass matrix is
  replaced by the lumped mass matrix, and the solves for $\alpha
  \concrete{R^*R} + \rho \concrete{W}_L$ and
  $\concrete{B}^T\concrete{B} + \rho \concrete{A}^T
  \concrete{W}_L^{-1} \concrete{A}$ are replaced by a small number of
  algebraic multigrid V-cycles. 
\end{itemize} 
For algebraic multigrid we use the root-node smoothed aggregation \cite{OlsonSchroderTuminaro11, VanekMandelBrezina96} method implemented in PyAMG \cite{BeOlSc2008}, with the default settings. One V-cycle is used for $\alpha \concrete{R^*R} + \rho \concrete{W}_L$, and three V-cycles are used for $\concrete{B}^T\concrete{B} + \rho \concrete{A}^T \concrete{W}_L^{-1} \concrete{A}$.

\subsection{Convergence comparison}

\begin{figure}[htbp]
\centering
\includegraphics[scale=0.65]{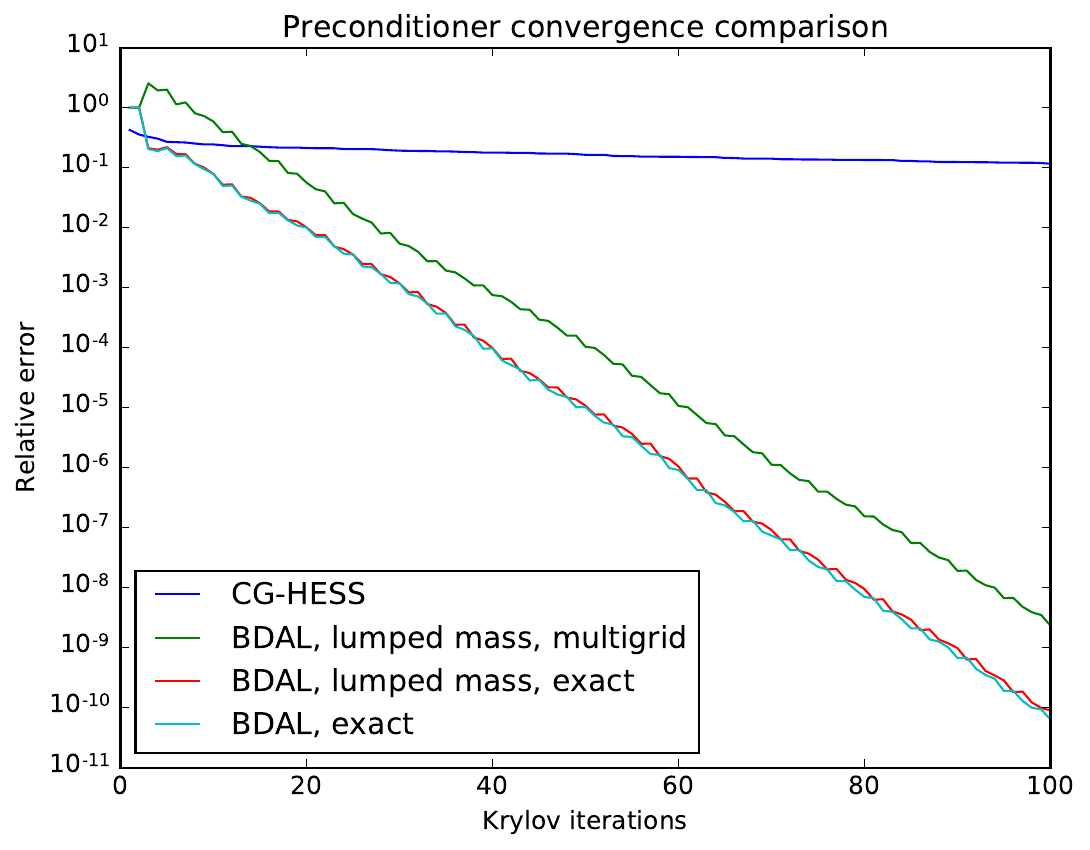}
\caption{Relative error in the parameter, $\nor{\concrete{q} - \concrete{q}_k}/\nor{\concrete{q}}$, for the high data Poisson source inversion problem, as a function of the number of Krylov iterations. The observation locations, regularization parameter, and mesh size are the same as in Figure \ref{fig:true_exact_obs} ($\nobs=2000$, $\alpha=10^{-8}$, $h=\sqrt{2} \cdot 10^{-2}$).}
\label{fig:convergence_comparison}
\end{figure}

\begin{figure}[htbp]
\centering
\includegraphics[scale=0.5]{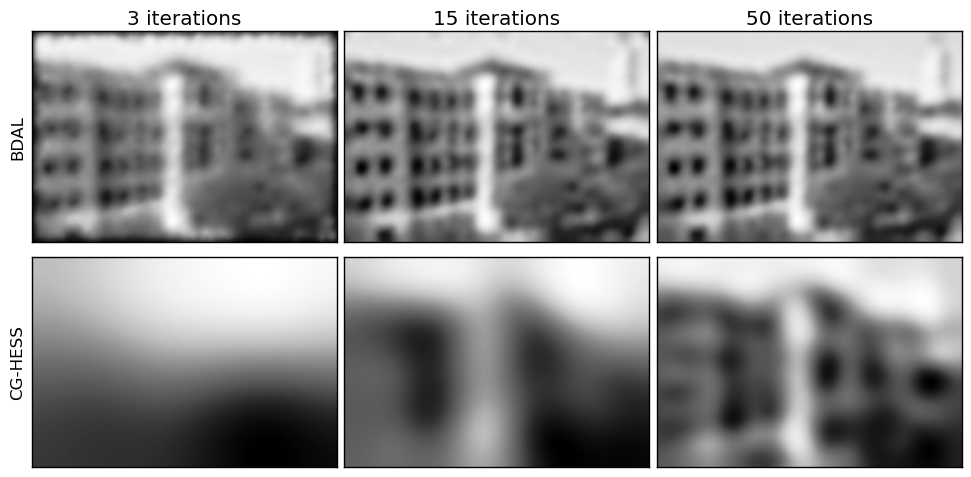}
\caption{Visual comparison of the 3rd, 15th, and 50th Krylov iterates
($\nobs=2000$,
  $\alpha=10^{-8}$, $h=\sqrt{2} \cdot 10^{-2}$). Top row:
  reconstruction using MINRES on the KKT system with our ``BDAL, lumped
  mass, exact'' preconditioner. Bottom row: reconstruction using CG on
  the reduced Hessian with regularization preconditioning.}
\label{fig:krylov_iterates_comparison}
\end{figure}

In Figure \ref{fig:convergence_comparison}, we show a convergence
comparison between between MINRES on the KKT system preconditioned by
our block diagonal augmented Lagrangian preconditioner, and conjugate
gradient on the reduced Hessian preconditioned by the regularization
term (CG-HESS). For our block diagonal augmented Lagrangian
preconditioner, we also show results for lumped mass and algebraic
multigrid approximations to the subsystems being solved. The
regularization, forward, and adjoint solves used for the reduced
Hessian solve are all performed exactly. The mesh size is $h=\sqrt{2}
\cdot 10^{-2}$, the number of observations is $2000$, and the
regularization parameter is $\alpha=10^{-8}$. Error is measured with
respect to the converged solution to the linear system
\eqref{eqn:discrete_kkt}, i.e.,
$\nor{\concrete{q}-\concrete{q}_k}/\nor{\concrete{q}}$. This allows us
to make a fair comparison between the reduced and full space methods.

In terms of Krylov iteration count, our preconditioner far outperforms
regularization preconditioning on the reduced Hessian. The error in
our method after three iterations is much less than the error after 50
iterations of regularization preconditioning on the reduced
Hessian. Performance with the lumped mass approximation is almost
identical to performance with exact solves. In the case with the multigrid
approximation, we see roughly the same asymptotic convergence rate as
the exact solve, but with a lag of $10$ to $20$ iterations. In our
numerical experiments we also observed that MINRES with our ``BDAL,
lumped mass, multigrid'' preconditioner takes considerably less time
per iteration than CG on the reduced Hessian, which is expected since
applying the reduced Hessian requires solving the forward and adjoint
equations to a high tolerance within each CG iteration. 

In Figure \ref{fig:krylov_iterates_comparison}, we see that the
reconstruction using the reduced Hessian starts off smooth, then
slowly includes information from successively higher frequency
parameter modes as the CG iterations progress. In contrast, our
preconditioner applied to the KKT system reconstructs low and high
frequency information simultaneously.

\subsection{Mesh scalability}

\begin{table}
\caption{Mesh scalability study for our ``BDAL, lumped mass, exact''
  preconditioner over a range of meshes. The table shows the number of
  MINRES iterations required to achieve parameter convergence to
  relative error $10^{-5}$.
The number of observations is $\nobs=2000$, and the regularization
parameter is $\alpha=10^{-8}$. The observation locations $x_\ind$ are
the same for all mesh sizes.}  \centering
\begin{tabular}{|c|c|c|}
\hline
$h$ & \# triangles & MINRES iterations \\
\hline
5.68e-02 & $1800$ & $51$ \\
2.84e-02 & $7200$ & $50$ \\
1.89e-02 & $16200$ & $51$ \\
1.41e-02 & $29000$ & $51$ \\
1.13e-02 & $45250$ & $51$ \\
9.44e-03 & $65100$ & $51$ \\
8.09e-03 & $88550$ & $51$ \\
7.07e-03 & $116000$ & $51$ \\
6.29e-03 & $146700$ & $51$ \\
5.66e-03 & $181000$ & $51$ \\
\hline
\end{tabular}
\label{tbl:mesh_scalability}
\end{table}

To test mesh scalability, we solve the Poisson source inversion
problem on a sequence of progressively finer meshes using MINRES with
our block diagonal augmented Lagrangian preconditioner.  The same
regularization parameter, $\alpha=10^{-8}$, and observation locations,
$\{x_\ind\}_{\ind=1}^{2000}$, are used for all meshes. The numbers of
iterations $k$ required to achieve a relative error of
$\nor{\concrete{q} - \concrete{q}_k}/\nor{\concrete{q}} < 10^{-5}$ are
shown in Table \ref{tbl:mesh_scalability}. All meshes are uniform
triangular meshes. The coarsest mesh has size $h=5.7 \cdot 10^{-2}$ with
$1,800$ triangles, and the finest mesh has $h=5.7 \cdot 10^{-3}$ with
$181,000$ triangles. To quantify the error, the exact solution 
$\concrete{q}$ was computed for each mesh using a sparse
factorization of the KKT matrix. All results are based on the
lumped mass approximation for mass matrices within the preconditioner.

The results clearly demonstrate mesh independence. The number of
MINRES iterations required remains essentially constant over a two
orders of magnitude increase in problem size, differing by at most one
iteration across all mesh sizes. 

\subsection{Regularization and data scalability}

\begin{figure}[htbp]
\centering
\includegraphics[scale=0.7]{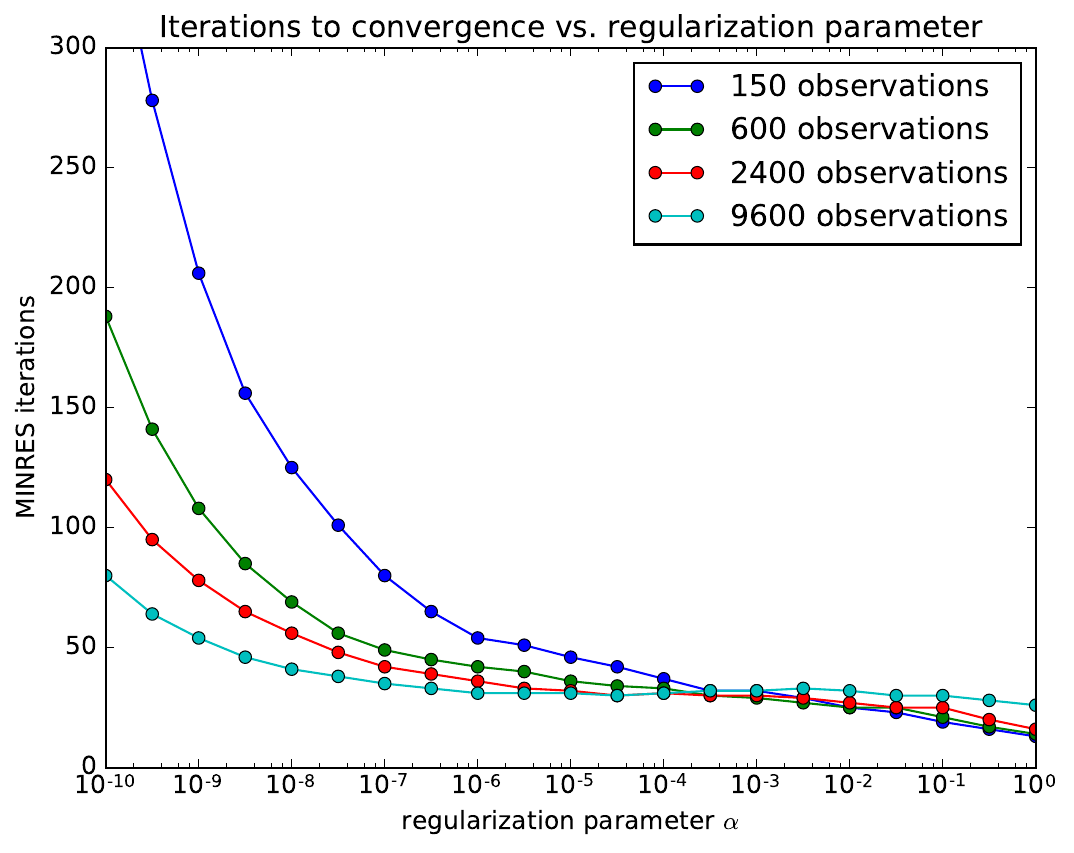}
\caption{Regularization and data scalability study for our ``BDAL,
  lumped mass, exact'' preconditioner. Plot shows the number of MINRES
  iterations $k$ required to achieve relative error $\nor{\concrete{q}
    - \concrete{q}_k}/\nor{\concrete{q}} < 1e-5$.} 
\label{fig:regularization_and_data_robustness}
\end{figure}

A data and regularization robustness study is shown in Figure \ref{fig:regularization_and_data_robustness}. The number of MINRES iterations $k$ required for the method to converge to an error $\nor{\concrete{q} - \concrete{q}_k}/\nor{\concrete{q}} < 10^{-5}$ is plotted for values of the regularization parameter in the range $\alpha \in [10^{-10}, 1.0]$, and number of observations $\nobs \in \{150, 600, 2400, 9600\}$. The mesh size is fixed at $h=\sqrt{2} \cdot 10^{-2}$, and for each value of $\nobs$, the observation locations, $x_\ind$, are fixed as the regularization parameter varies. 

The overall performance of the preconditioner is relatively steady
over a broad range of values of $\alpha$ and $\nobs$. The performance
of the method does decrease as the regularization parameter goes to
zero for a fixed number of observations (upper left, Figure
\ref{fig:regularization_and_data_robustness}). However, the
combination of small regularization parameter and small number of
observations corresponds to the under-regularized regime, which we
would not find ourselves in for an appropriately regularized problem.
As the number of observations increases, the performance of the method
improves in the small regularization regime while slightly worsening
in the large regularization (over-regularized) regime, as suggested by
our theory. This behavior is consistent with a data scalable method:
one can take small values for the regularization parameter if that
choice is supported by the data available in the problem. 

\section{Conclusion}

Traditional methods for solving linear inverse problems either scale
poorly with increasing data and decreasing regularization, or are
restricted to specific forms of regularization that may not be
appropriate for the inverse problem at hand, or apply only to very
specific problems. To overcome these limitations, we proposed a
preconditioner based on a block diagonal approximation to the
augmented Lagrangian KKT operator. We proved bounds on the condition
number of the preconditioned system in an abstract setting,
specialized the analysis to the case of source inversion problems with
spectral filtering regularization, and tested the preconditioner
numerically on a Poisson source inversion problem with highly
informative data and small regularization parameter. Our analysis and
numerical results indicate that the preconditioner is mesh and data
scalable when the regularization does not over-penalize highly
informed parameter modes and does not under-penalize uninformed
modes.


\section*{Acknowledgements}

We thank James Martin and Toby Isaac for helpful discussions, and James Martin for editing suggestions on an early draft of this paper. We thank the anonymous reviewers for their helpful comments. We would in particular like to thank one of the reviewers for bring the improved Brezzi theory bound in \cite{Krendl13} to our attention, as this allowed the constant in Theorem \ref{thm:cond_bound} to be decreased from $2 + 2\sqrt{2}$ to $3$.

\FloatBarrier

\bibliographystyle{siamplain}
\bibliography{ccgo}
\end{document}